\documentclass[12pt, reqno, a4paper]{amsart}
\usepackage[margin=1.2in]{geometry}
\numberwithin{equation}{section}
\usepackage{amssymb,amsfonts,amsthm}
\usepackage[utf8]{inputenc}
\usepackage{listings}
\usepackage{bm}
\usepackage[hyperpageref]{backref}
\usepackage{esint}
\usepackage{bigints}
\usepackage{amsthm}
\usepackage{enumitem}
\usepackage{color}
\usepackage[bookmarks]{hyperref}
\usepackage{siunitx}
\usepackage{hyperref}
\usepackage[T1]{fontenc}
\theoremstyle{plain}
\usepackage{tikz}
\usetikzlibrary{shapes}
\usepackage{caption} 

\usepackage{mathtools}

\allowdisplaybreaks[4] 
\usepackage[T1]{fontenc}
\usepackage{amsthm}

\newtheoremstyle{myremark}{10pt}{10pt}{}{}{\bfseries}{.}{.5em}{}

\newtheorem{definition}{Definition}[section]
\newtheorem{theorem}{Theorem}[section]
\newtheorem{prop}{Proposition}[section]
\newtheorem{lemma}{Lemma}[section]

\usepackage{hyperref}

\allowdisplaybreaks[4]

\begin{document}

\title[ Full Discretization of Stochastic Semilinear Schr\"{o}dinger equation]{Full Discretization of Stochastic Semilinear Schr\"{o}dinger equation driven by multiplicative Wiener noise}
\address{Department of Mathematics and Statistics, Indian Institute of Technology Kanpur, Kanpur-208016, India}

\author{Suprio Bhar}
\email{Suprio Bhar: suprio@iitk.ac.in}
\author{Mrinmay Biswas}
\email{Mrinmay Biswas: mbiswas@iitk.ac.in}
\author{Mangala Prasad}
\email{Mangala Prasad: mangalap21@iitk.ac.in}

	

\date{}

\dedicatory{}

\begin{abstract}
	In this article, we have analyzed the full discretization of the Stochastic semilinear Schr\"{o}dinger equation in a bounded convex polygonal domain driven by multiplicative Wiener noise. We use the finite element method for spatial discretization and the stochastic trigonometric method for time discretization and derive a strong convergence rate with respect to both parameters (temporal and spatial). Numerical experiments have also been performed to support theoretical bounds. 
	\end{abstract}

\maketitle

\section{Introduction and Main Results}\label{s1}
	We study the full discretization of stochastic semilinear Schrodinger equation driven by multiplicative  noise,
	\begin{equation}\label{eqn1.1}
		 \begin{split}
       & du+ i\Delta udt= g(u)dt +f_1(u)dW_1+if_2(u) dW_2\, \quad \text{in} \quad (0,\infty) \times \mathcal{O},\\
       & u=0 \quad \text{in}  \quad (0,\infty) \times \partial \mathcal{O},\\
       & u(0,x)=u_0 \quad \text{in} \quad \mathcal{O}.
    \end{split}
	\end{equation}
	where $\mathcal{O} \subset{\mathbb{R}^d}$, $d=1,2,3$ is a bounded convex polygonal domain with boundary $\partial \mathcal{O},$ and $\{W_j(t)\}_{t \geq 0}$ for $j=1,2$ be two $L^2(\mathcal{O})$-valued Wiener processes on a filtered probability space $\big(\Omega,\mathcal{F},P,\{\mathcal{F}_t\}_{t\geq 0}\big)$ with respect to the filtration $\{\mathcal{F}_t\}_{t\geq 0}$.iIn our analysis the Wiener process need not be independent. The smoothness assumptions for the nonlinearities $g$, $f_1$, and $f_2$ will be specified below. The initial date  $u_0$ is a $\mathcal{F}_0$-measurable random variable. We shall discretize with a stochastic trigonometric method in time and a linear finite element method in space.

    We refer the reader to the introduction of \cite{finite} for the relevant literature on the spatial discretization of stochastic Schr\"{o}dinger equations. A full discretization (time and space) of stochastic wave and heat equations with multiplicative noise have been studied in the literature; see for example \cite{cohen2016, davidcohen16, YYanSiam5}. Analytical and physical properties of solutions of deterministic linear and semi-linear Schr\"{o}dinger equation and its applications have been extensively studied in the literature \cite{courantL,dautray} and references therein. The existence and uniqueness of a solution of stochastic Scr\"{o}dinger equations have been studied, see for example \cite{debussche1999,debussche2003}. To our knowledge, the discretization of stochastic semilinear Schr\"{o}dinger equation in time and space with multiplicative noise is a new result.

    In this paper, we prove the mean square convergence for the full discretization of stochastic linear Schr\"{o}dinger equation with additive noise. Furthermore, we study the full discretization of stochastic semilinear  Schr\"{o}dinger equation with multiplicative noise. 
 \subsection{\bf Setup and Main results: Full Discretization for the Linear Case } 
	We first study the full discretization of stochastic Schr\"{o}dinger equation driven by additive noise
	\begin{equation}\label{eqn1.2}
		\begin{split}
			&du+i \Delta u\,dt =dW_1 +i\, dW_2,  \quad \text{in}\quad ( 0,\infty)\times  \mathcal{O}, \\
			&u=0, \quad \text{in}\quad  (0,\infty)\times  \partial \mathcal{O},\\
			&u(0,x)=u_0(x), \quad \text{in}\quad \mathcal{O}.
		\end{split}
	\end{equation}
    Let us put $u_1=Re(u),\, u_2=Im(u)$ and $u_{0,1}=Re(u_0),\, u_{0,2}=Im(u_0)$. We use the semigroup framework to study the full discretization of stochastic linear Schr\"{o}dinger equation. As above, equation \eqref{eqn1.2} can be written in a system form as
	\begin{equation}\label{Linear}
		d {\begin{bmatrix}
				u_1\\
				u_2
		\end{bmatrix}}={\begin{bmatrix}
				0 & -A\\
				A & 0
		\end{bmatrix}} 
		\begin{bmatrix}
			u_1\\
			u_2
		\end{bmatrix}
		dt+ \begin{bmatrix}
			dW_1\\
			dW_2
		\end{bmatrix},\quad t>0, \quad \begin{bmatrix}
			u_1(0)\\
			u_2(0)
		\end{bmatrix}=\begin{bmatrix}
			u_{0,1}\\
			u_{0,2}
		\end{bmatrix},
	\end{equation}
  where $A$ is Laplace operator defined in Appendix \ref{s2.1}. The System \eqref{Linear} can be written in an abstract form as
	\begin{equation}\label{eqn1.3}
		dX(t)=\mathbb{A}X(t)dt+dW(t), \hspace{.2cm} t>0; \hspace{.2cm} X(0)=X_0,
	\end{equation}
	where $ X= \begin{bmatrix}
		u_1 \\
		u_2
	\end{bmatrix}$, $X_0:=\begin{bmatrix}
		u_{0,1}\\
		u_{0,2}
	\end{bmatrix}, \, \mathbb{A}=\begin{bmatrix}
		0 & -A \\
		A  & 0
	\end{bmatrix}$ and $dW:=\begin{bmatrix}
			dW_1\\
			dW_2
		\end{bmatrix}.$ 
	In this framework, the weak solution of equation \eqref{eqn1.3} is given by
	\begin{equation}\label{Slinear}
		X(t)=e^{t\mathbb{A}}X_0+ \int_0^t e^{(t-\tau)\mathbb{A}}dW(\tau)\, ,
	\end{equation}
    where \[e^{t\mathbb{A}}=
	\begin{bmatrix}
		C(t) & -S(t)\\
		S(t) & C(t)
	\end{bmatrix},\quad t\geq 0.
 \]
 The definition of $C(t) \text{ and }S(t)$ are defined in Appendix  \ref{s2.1}. It is well known that the solution \eqref{Slinear} satisfies for some $C=C_{\mathcal{O}}>0,$
	\[
		|||X(t)|||_{L^2(\Omega , \mathbf{H}^{\theta})} \leq C \left( |||X_0|||_{L^2(\Omega , \mathbf{H}^{\theta})} + t^{1/2}\left( \|A ^{\theta /2}Q_1^{1/2}\|_{HS}+\|A ^{\theta /2}Q_2^{1/2}\|_{HS}\right)\right),\quad t \geq 0.
	\]
	We now consider the finite element approximations of the stochastic linear Schr\"{o}odinger equation \eqref{eqn1.3}. We discretize the spatial variables with a standard piecewise linear finite element method. The spatially discrete analog of
	the equation \eqref{eqn1.3} is to find $X_h(t)=(u_{h,1}(t),u_{h,2}(t))^\text{T} \in V_h \times V_h$ (definition of $V_h$ is define in Appendix \ref{s2.2}) such that 
	\begin{equation}\label{eqn1.4}
		dX_h(t)=\mathbb{A}_hX_h(t)dt+ \mathcal{P}_h dW(t), \hspace{.3cm} t>0, \hspace{.2cm } X_h(0)=X_{0,h},
	\end{equation}
    where $\mathbb{A}_h=\begin{bmatrix}
		0 & -A _h\\
		A _h & 0
	\end{bmatrix}$ and $\mathcal{P}_h$ is projection operator define in Appendix \ref{s2.2}.
	  The unique mild solution of \eqref{eqn1.4} is given by 
	\begin{equation}\label{Dslinear}
		X_h(t)=e^{t\mathbb{A}_h}X_{0,h}+\int_0^t e^{(t-\tau)\mathbb{A}_h}\mathcal{P}_hdW(\tau),\quad t\geq 0.
	\end{equation}
    Here $\mathbb{A}_h$ is a generator of $C_0$ semigroup $e^{t\mathbb{A}_h}$ and it is given by  
 \[e^{t\mathbb{A}_h}=
	\begin{bmatrix}
		C_h(t) & -S_h(t)\\
		S_h(t) & C_h(t)
	\end{bmatrix},\quad t\geq 0,
 \]
	where $C_h(t)=\cos{(tA _h)},\hspace{.2cm} S_h(t)=\sin{(tA _h )}.$ 
	For example, similar to the infinite-dimensional case, using $\{(A _{h,j},\phi _{h,j})\}_{j=1}^{N_h}$, the
	orthonormal eigenpairs of the discrete Laplacian $A _h$, with $N_h = \text{dim}(V_h)$, we have 
 for $t\geq 0,$
	\[
	C_h(t)v_h=\cos{(tA _h)}v_h=\sum_{j=1}^{N_h}\cos{(t\lambda_{h,j})}(v_{h,j},\phi_{h,j})\phi_{h,j}, \hspace{.2cm} v_h\in V_h.
	\]
    We present the following theorem on the convergence of spatial approximation.
\begin{theorem}[Theorem 1.4 
 in \cite{finite}]\label{thm1}
Let $\theta \in [0,2]$, the covariance operators $Q_i,\,i=1,2$ satisfy $$\|A ^{\theta /2}Q_1^{1/2}\|_{HS}+\|A ^{\theta /2}Q_2^{1/2}\|_{HS}<\infty,$$ $X_0=(u_{0,1},u_{0,2})^\text{T}$ and $X_{0,h}=(u_{h,0,1},u_{h,0,2})^\text{T}=(\mathcal{P}_h u_{0,1},\mathcal{P}_h u_{0,2})^\text{T}$. Let $X=(u_1,u_2)^\text{T}$ and $X_h=(u_{h,1},u_{h,2})^\text{T}$ be given by $\eqref{Slinear} \text{ and } \eqref{Dslinear}$ respectively. Then, the following estimates hold for $t \geq 0$:
		\begin{equation*}
		\begin{split}
			\|u_{h,1}(t)-u_1(t)\|_{L^2(\Omega;\dot{H}^0)} \leq C_th^{\theta} \left(|||X(0)|||_{L^2(\Omega,\mathbf{H}^{\theta})} + \|A ^{\theta /2}Q_1^{1/2}\|_{HS}+\|A ^{\theta /2}Q_2^{1/2}\|_{HS}\right),\\
			\|u_{h,2}(t)-u_2(t)\|_{L^2(\Omega;\dot{H}^0)}
			\leq C_th^{\theta} \left(|||X(0)|||_{L^2(\Omega,\mathbf{H}^{\theta})} + \|A ^{\theta /2}Q_1^{1/2}\|_{HS}+\|A ^{\theta /2}Q_2^{1/2}\|_{HS}\right),
		\end{split}
		\end{equation*}
		 where $C_t$ is an increasing function in time.
	\end{theorem}
  To discretize the finite element approximation \eqref{eqn1.4} in time, one is often interested in using explicit methods with large step sizes. Let $0=t_0< \,\cdots\,< T_{N_t}$ be a uniform partition of the time interval $[0,T_N]$ with the time step $k=1/N_t$ and time subintervals $I_n=(t_{n-1},t_n),\, n=0,\,\cdots\,N_t$. Then the stochastic trignometric method is given, for $n=0,\, \cdots, N_t$ 
       $U^{n+1}=e^{k\mathbb{A}_h}U^n+e^{k\mathbb{A}_h}\mathcal{P}_h \Delta W^n$ that is 
\begin{equation} \label{eqn1.5}
   \begin{bmatrix}
       U_1^{n+1} \\
       U_2^{n+1}
   \end{bmatrix} =
   \begin{bmatrix}
     C_h(k) & -S_h(k)  \\
     S_h(k) & C_h(k)
   \end{bmatrix}
   \begin{bmatrix}
       U^n_1\\
       U^n_2
   \end{bmatrix}+ \begin{bmatrix}
     C_h(k) & -S_h(k)  \\
     S_h(k) & C_h(k)
   \end{bmatrix}
   \begin{bmatrix}
       \mathcal{P}_h \Delta W^n_1\\
        \mathcal{P}_h \Delta W^n_2
   \end{bmatrix}.
\end{equation}
where $\Delta W_j^n=W_j(t_{n+1})-W_j(t_{n})$ for $j=1,2 $ denotes the Wiener increments and $U^0_1=u_{h,1}(0),\, U^0_2=u_{h,2}(0)$. Therefore here we get an approximation $U^n_j \approx u_{h,j}(t_n)$
 of the exact solution of our finite element problem at the discrete times $t_n=nk$. 
 
 We formulate an error estimate for the full discretization. 
 \begin{theorem}\label{thm2}
     Consider the numerical discretization of equation \eqref{eqn1.4} using the stochastic trigonometric scheme \eqref{eqn1.5}, where the time step size is 
$k$. The strong global errors of this numerical scheme satisfy the following estimate. If the condition\[\|A^{\theta /2}Q_1^{1/2}\|_{HS}+ \|A^{\theta /2}Q_2^{1/2}\|_{HS} < \infty \] holds $\text{ for } \theta \in [0,2]$, then the numerical solution satisfies the error bounds
      \begin{equation*}
      \begin{split}
          \|U^n_1-u_{h,1}(t_n)\|_{L_2(\Omega, \dot{H}^0)} \leq Ck^{\theta /2}(\|A^{\theta /2}Q_1^{1/2}\|_{HS}+ \|A^{\theta /2}Q_2^{1/2}\|_{HS} )\\
          \|U^n_2-u_{h,2}(t_n)\|_{L_2(\Omega, \dot{H}^0)} \leq Ck^{\theta /2}(\|A^{\theta /2}Q_1^{1/2}\|_{HS}+ \|A^{\theta /2}Q_2^{1/2}\|_{HS} ).
          \end{split}
      \end{equation*}
  Here, $C=C(T)$ is a constant that depends on $T$ but remains independent of the parameters $ t_n=nk \leq T$ and $n$ provided that $t_n=nk \leq T$.
\end{theorem}
The combination of Theorem \ref{thm1} and Theorem \ref{thm2} yields the following result for the full discretization.
\begin{theorem}\label{thm3}
     Consider the numerical solution of \eqref{eqn1.2} by the finite element method in space with a maximal mesh size $h$ and the numerical scheme \eqref{eqn1.5} with a time step size $k$ on the time interval $[0, T]$. Let $X=(u_1,u_2)$ and $X_h=(u_{h,1},u_{h,2})$ be the solution of \eqref{eqn1.3} and \eqref{eqn1.4} respectively and the initial random variable $X_0=(u_1(0),u_2(0))$ satisfies 
       $|||X_0|||_{L_2(\Omega,\mathbf{H}^{\theta})} < \infty$.
      If $u_{h,1}=\mathcal{P}_hu_1(0),u_{h,2}(0)=\mathcal{P}_hu_2(0) \text{ and } \|A^{\theta /2}Q_1^{1/2}\|_{HS}+ \|A^{\theta /2}Q_2^{1/2}\|_{HS} < \infty $ holds for $\theta \in [0,2]$, then the numerical solution satisfies the error bounds
      \[
      \begin{split}
          \|U^n_1-u_1(t_n)\|_{L_2(\Omega, \dot{H}^0)} \leq C(T)(h^{\theta}+k^{\theta /2})(\|A^{\theta /2}Q_1^{1/2}\|_{HS}+ \|A^{\theta /2}Q_2^{1/2}\|_{HS} )\quad \text{ for } \theta \in [0,2],\\
            \|U^n_1-u_1(t_n)\|_{L_2(\Omega, \dot{H}^0)} \leq C(T)(h^{\theta}+k^{\theta /2})(\|A^{\theta /2}Q_1^{1/2}\|_{HS}+ \|A^{\theta /2}Q_2^{1/2}\|_{HS} ) \quad \text{ for } \theta \in [0,2].
      \end{split}
      \]
        Here, $C=C(T)$ is a constant that depends on $T$ but remains independent of the parameters $ t_n=nk \leq T$ and $n$ provided that $t_n=nk \leq T$.
      \end{theorem}
\subsection{\bf Main results: Full Discretization for the Semilinear Case} 
We consider the numerical discretization of stochastic semilinear Schr\"{o}dinger equation \eqref{eqn1.1} driven by multiplicative noise.
The abstract form of \eqref{eqn1.1} is given by 
\begin{equation*}
    d \begin{bmatrix}
        u_1 \\
        u_2
    \end{bmatrix}=\begin{bmatrix}
        0 &-A\\
        A &0
    \end{bmatrix}
    \begin{bmatrix}
        u_1\\
        u_2
    \end{bmatrix}dt + \begin{bmatrix}
        g_1(u_1,u_2) \\
        g_2(u_1,u_2)
    \end{bmatrix}dt +\begin{bmatrix}
        f_1(u_1,u_2) & 0\\
        0 & f_2(u_1,u_2)
    \end{bmatrix} dW \quad \text{for } t>0
\end{equation*}
with initial condition $ \begin{bmatrix}
        u_1(0)\\
        u_2(0)
    \end{bmatrix}=\begin{bmatrix}
        u_{0,1}\\
        u_{0,2}
    \end{bmatrix}$ and $dW:=\begin{bmatrix}
    dW_1\\
    dW_2
\end{bmatrix}$.
The equation \eqref{eqn1.1} can be written as 
\begin{equation} \label{eqn1.6}
    dX(t)=(\mathbb{A}X(t)+G(X(t)))dt+ F(X(t))dW(t),\quad t>0,\, X(0)=X_0,
\end{equation}
where $X(t)=(u_1(t),u_2(t))^T$ for $t\geq 0 $, $X_0=(u_{0,1},u_{0,2})^T$,\\
 $ F(X(t))=\begin{bmatrix}
     f_1(u_1(t),u_2(t)) & 0\\
     0 & f_2(u_1(t),u_2(t))
 \end{bmatrix}$,\, $G(X(t))=(g_1(u_1(t),u_2(t)),\,g_2(u_1(t),u_2(t)))^T$ and 
$\mathbb{A}$ generates the semigroup $\{e^{t\mathbb{A}}\}_{t\geq 0}$.
The weak solution of \eqref{eqn1.6} is given by 
\begin{equation}\label{eqn1.7}
    X(t)=e^{t\mathbb{A}}X_0+ \int_0^t e^{(t-\tau)\mathbb{A}}G(X(\tau))d\tau+\int_0^t e^{(t-\tau)\mathbb{A}}F(X(\tau))dW(\tau),\quad t\geq 0.
\end{equation}
The semi-discretization of \eqref{eqn1.6} is analogue of 
\begin{equation}\label{eqn1.8}
    dX_h(t) =\mathbb{A}_hX_h(t)dt+\mathcal{P}_hG(_h(t))dt+\mathcal{P}_hF(X_h(s))dW(s),\quad  t>0, \quad X_h(0)=X_{h,0}.
\end{equation}
For $t \geq 0$, the mild solution of \eqref{eqn1.8} is given  by 
\begin{equation}\label{eqn1.9}
    X_h(t)= e^{t\mathbb{A}_h}X_{h,0}+\int_0^t  e^{(t-\tau)\mathbb{A}_h}\mathcal{P}_hG(X_h(\tau))d\tau+\int_0^t e^{(t-\tau)\mathbb{A}_h}\mathcal{P}_hF(X_h(\tau))dW(\tau).
\end{equation}
To ensure the existence and uniqueness of the solution to problem \eqref{eqn1.1} we assume that the initial condition satisfies $$u_0=(u_{0,1},\,u_{0,2})^T \in L_2(\Omega, \mathbf{H}^{\theta}) \text{ for } \theta \in [0,2].$$ Furthermore, we impose the following conditions on the functions $g=(g_1,g_2):\mathbf{H} \to \mathbf{H}$ and $f=(f_1,f_2):\mathbf{H} \to \mathcal{L}_2(\mathbf{H},\mathbf{H})$:

\begin{equation}\label{assum}
\begin{split}
    &\left\|A^{\theta/ 2}g_1(u_1(s),u_2(s)) \right\| + \left\|A^{\theta/ 2}g_2(u_1(s),u_2(s))\right \| \leq L_g(\left\|u_1(s)\right\|_{\theta }+\|u_2(s)\|_{\theta }+1) ,\\
& \left\|g_1(u_1,u_2)-g_1(u_1^{\prime},u_2^{\prime})\right\|+ \left\|g_2(u_1,u_2)-g_2(u_1^{\prime},u_2^{\prime})\right\|\leq K_g(\|u_1-u_1^{\prime}\|+\|u_2-u_2^{\prime}\|)  ,\\ 
&\left\|A^{\theta/2}f_1(u_1(s),u_2(s))Q^{1/2}_1\right\|_{HS}+\left\|A^{\theta/2}f_2(u_1(s),u_2(s))Q^{1/2}_2\right\|_{HS}\\
&\leq L_f(\|u_1(s)\|_{\theta}+\|u_1(s)\|_{\theta}+1),\\
 &\left\|(f_1(u_1,u_2)-f_1(u_1^{\prime},u_2^{\prime}))Q_1^{1/2}\right\|_{HS}+ \left\|f_2(u_1,u_2)-f_2(u_1^{\prime},u_2^{\prime}))Q^{1/2}_2\right\|_{HS}\\
 &\leq K_f(\|u_1-u_1^{\prime}\|+\|u_2-u_2^{\prime}\|) 
 \end{split}   
\end{equation}
for all $u_1,\,u_2,\,u_1^{\prime}, \text{ } u_2^{\prime} \in L_2(\mathcal{O})$ and $\theta \in [0,2]$ . Throughout this paper, the inner product $(\cdot,\cdot)$ and norm $\|\cdot\|=\|\cdot\|_{L_2(\mathcal{O})}$ are defined on $L_2(\mathcal{O})$. The norm $\|\cdot\|_{HS}$ and the Hilbert space $\mathbf{H}$ is defined in Appendix \ref{s2.1}.
\begin{prop}\label{prop1}
   Let $u=(u_1,u_2)^T$ be the solution of \eqref{eqn1.1} with initial condition $u_0=(u_{0,1}\, u_{0,2})^T\text{ in } L_2(\Omega,\, \mathbf{H}^{\theta})$ for $0 \leq \theta \leq 2$. Assume that the functions  $ f=(f_1,\, f_2) \text{ and }g=(g_1,g_2)$ satisfy required conditions given in \eqref{assum} for $\theta \in [0,2]$. Then, for any $T>0$, we have the uniform bound 
   \[
   \sup_{0 \leq t \leq T} \mathbb{E}[\|u_1(t)\|_\theta^2+\|u_2(t)\|_\theta^2] \leq C.
   \]
Furthermore, for any $0 \leq s \leq t \leq T$, the following temporal continuity estimates hold:
 \[
\begin{split}
    \mathbb{E}\|u_1(t)-u_1(s)\|^2&\leq C_1|t-s|^{\theta}\left(\left\|u_{0,1}\right\|^2_{\theta}+\left\|u_{0,2}\right\|^2_{\theta}+\sup_{0 \leq t \leq T}\mathbb{E}\left(\|u_1(t)\|^2_{\theta}+\|u_2(t)\|^2_{\theta}+1\right)\right)\\
    &+C_2|t-s|\left(\sup_{0 \leq t \leq T}\mathbb{E}[\|u_1(t)\|^2_{\theta}+\|u_2(t)\|^2_{\theta}+1]\right),\\
     \mathbb{E}\|u_2(t)-u_2(s)\|^2&\leq C_3|t-s|^{\theta}\left(\left\|u_{0,1}\right\|^2_{\theta}+\left\|u_{0,2}\right\|^2_{\theta}+\sup_{0 \leq t \leq T}\mathbb{E}\left(\|u_1(t)\|^2_{\theta}+\|u_2(t)\|^2_{\theta}+1\right)\right)\\
    &+C_4|t-s|\left(\sup_{0 \leq t \leq T}\mathbb{E}[\|u_1(t)\|^2_{\theta}+\|u_2(t)\|^2_{\theta}+1]\right).
\end{split}
\] 
Here the constants $C,\, C_1,\, C_2,\,C_3,\, \text{and } C_4$ depend  only on $T$ and are independent of $\theta$. 
\end{prop}
The explicit time discretization of the finite element solution \eqref{eqn1.8} of the stochastic Schr\"{o}dinger equation using a stochastic  trigonometric method with step size $k$, we get 
\[
U^{n+1}= e^{k\mathbb{A}_h}U^n+e^{k\mathbb{A}_h}\mathcal{P}_hG(U^n).k+e^{k\mathbb{A}_h}\mathcal{P}_hF(U^n) \Delta W^{n}
\]
that is 
\begin{equation} \label{eqn1.10}
\begin{split}
    \begin{bmatrix}
        U_1^{n+1}\\
        U_2^{n+1}
    \end{bmatrix}=
    \begin{bmatrix}
        C_h(k) & -S_h(k)\\
        S_h(k) & C_h(k)
    \end{bmatrix} 
    \begin{bmatrix}
        U_1^n \\
        U^n_2
    \end{bmatrix}&+ \begin{bmatrix}
        C_h(k) & -S_h(k)\\
        S_h(k) & C_h(k)
    \end{bmatrix}
    \begin{bmatrix}
        \mathcal{P}_hg_1(U_1^n,\, U^n_2)\\
        \mathcal{P}_hg_2(U_1^n,\, U^n_2)
    \end{bmatrix}k\\
    &+
    \begin{bmatrix}
        C_h(k) & -S_h(k)\\
        S_h(k) & C_h(k)
    \end{bmatrix} 
    \begin{bmatrix}
        \mathcal{P}_hf_1(U^n_1,U^n_2)\Delta W_1^n\\
        \mathcal{P}_hf_2(U^n_1,U^n_2)\Delta W_2^n 
    \end{bmatrix}
    \end{split}
\end{equation}
where $\Delta W^n_j=W_j(t_{n+1})-W_j(t_{n})$ is Wiener increments for $j=1,2$. Here we get an approximation $U_j^n \approx u_{h,j}(t_n)$ of the exact finite element problem at the discrete times $t_n=nk$ for $j=1,2$. Further, we get a recursion formula
\begin{equation}\label{rformula}
U^n=e^{t_n\mathbb{A}_h}U^0+ \sum_{j=0}^{n-1}e^{(t_n-t_j)\mathbb{A}_h}\mathcal{P}_hG(U^j)k+\sum_{j=0}^{n-1}e^{(t_n-t_j)\mathbb{A}_h}\mathcal{P}_hF(U^j) \Delta W^j.
    \end{equation}
Here we formulate an error estimate for the full discretization of stochastic semilinear Sch\"{o}dinger equation.
\begin{theorem} \label{thm4}
   Consider the numerical discretization of the stochastic semilinear Schr\"{o}dinger equation with multiplicative noise \eqref{eqn1.1} over a compact time interval  $[0, T],\,\text{where } T>0$. The spatial domain is discretized using the linear finite element method, while the temporal discretization is handled via the stochastic trigonometric scheme \eqref{eqn1.10}. Assume that the initial condition is given by 
 $u_0=(u_{0,1}\, u_{0,2})^T\text{ with  }u_0\in  L_2(\Omega,\, \mathbf{H}^{\theta})$ and  and that the functions 
$f \text{ and }g$ satisfy the conditions in \eqref{assum} for 
$\theta \in [0,2]$. Then, for any $t_n\in [0,T]$, the mean square error estimates are given by
    \[
    \begin{split}
        \|U^n_1-u_1(t_n)\|_{L_2(\Omega,\dot{H}^0)}\leq  C(h^{\theta}+k^{min(\theta/2,1/2)}) \quad \text{ for } \theta \in [0,2],\\
         \|U^n_1-u_1(t_n)\|_{L_2(\Omega,\dot{H}^0)}\leq  C(h^{\theta}+k^{min(\theta/2,1/2)}) \quad \text{ for } \theta \in [0,2].
    \end{split}
    \]
Here, the constant 
$C$ depends on 
$T,\, f,\, \text{ and } g$ but remains independent of the discretization parameters 
$k \text{ and } h.$ Additionally, it is assumed that the time step size $k$ satisfies $k<1$.
\end{theorem}
	\subsection{Organization of the paper}
	The paper is organized as follows. Section \ref{s1} introduces the model and states the main results. Section \ref{prf_main} contains the proofs of the theorems and propositions, which are stated in Section \ref{s1}. In Section \ref{S_4}, we present numerical examples and corresponding figures that illustrate the convergence rate. Finally, in the appendices, we provide preliminaries on Hilbert-Schmidt operators, infinite-dimensional Wiener noise, semigroup formulations, and basic finite element estimates. 
	\section{Proofs of Main results}\label{prf_main}
	In this section, we discuss the proof of Theorems \ref{thm2}-\ref{thm4} and Proposition \ref{prop1}. We first prove the following lemmas before proving the theorems and propositions.
    \begin{lemma}\label{lemma0}
        Let $\theta \in [0,2]. $ Then, for all $0\leq s \leq t \leq T$, the following H\"{o}lder continuity estimates hold for the operators:
         \begin{align}
        &\|(S_h(t)-S_h(s))A^{-\theta/2}_h\|\leq C|t-s|^{\theta /2}, \label{eqn3.1}\\
       & \|(C_h(t)-C_h(s))A^{-\theta/2}_h\|\leq C|t-s|^{\theta /2},\label{eqn3.2}\\
        & \|(S(t)-S(s))A^{-\theta /2}\|\leq C|t-s|^{\theta /2},\label{eqn3.3}\\
         &\|(C(t)-C(s))A^{-\theta /2}\|\leq C|t-s|^{\theta /2} \label{eqn3.4}.
    \end{align}
    \end{lemma}
\begin{proof}
   We prove the estimate \eqref{eqn3.1}. The other estimate has similar proof. For $\theta =0 \text{ and } v_h \in  V_h$, we use boundedness  of $S_h(t)$.
    \[
    \|(S_h(t)-S_h(s))v_h\|_{L_2(\mathcal{O})}\leq 2\|v_h\|_{h,0}.
    \]
    For $\theta =2 ,\, v_h \in V_h \text{ and } 0\leq s\leq t\leq T$, we  have 
    \[
    \begin{split}
    (S_h(t)-S_h(s))v_h& =\int_s ^t \frac{d}{dr}(S_h(r))v_hdr\\
    &=\int _s ^t C_h(r)A_hv_hdr\\
    \|(S_h(t)-S_h(s))v_h\|_{L_2(\mathcal{O})}& \leq 
    |t-s\|A_h v_h\|_{L_2(\mathcal{O})}=|t-s|\|v_h\|_{h,2}.
    \end{split}
    \]
    By interpolation argument, for $ 0\leq s\leq t\leq T$ we get 
    \[
    \|(S_h(t)-S_h(s))v_h\|_{L_2(\dot{H}^0)} \leq C|t-s|^{\theta /2}\|v_h\|_{h,\theta}, \hspace{2mm} v_h \in V_h, \hspace{2mm } \theta \in [0,2]
    \]
 which is \eqref{eqn3.1}. Other proofs are the same as \eqref{eqn3.1}.
\end{proof}
  \begin{lemma}\label{lemma1}
     Let us denote $d^n=(d_1^n,d_2^n)$ by 
     \[
     \begin{split}
         d_1^n:=\int_{t_n}^{t_{n+1}} C_h(t_{n+1}-\tau)\mathcal{P}_hdW_1(\tau)-\int_{t_n}^{t_{n+1}} S_h(t_{n+1}-\tau)\mathcal{P}_hdW_2(\tau)\\
         -\left(C_h(k)\mathcal{P}_h \Delta W_1^n-S_h(k)\mathcal{P}_h \Delta W_2^n\right),\\
          d_2^n:=\int_{t_n}^{t_{n+1}} S_h(t_{n+1}-\tau)\mathcal{P}_hdW_1(\tau)+\int_{t_n}^{t_{n+1}} C_h(t_{n+1}-\tau)\mathcal{P}_hdW_2(\tau)
          \\-\left(S_h(k)\mathcal{P}_h \Delta W_1^n+C_h(k)\mathcal{P}_h \Delta W_2^n\right).
     \end{split}
     \]     
     If  $\|A^{\theta /2}Q_1^{1/2}\|_{HS}+ \|A^{\theta /2}Q_2^{1/2}\|_{HS} < \infty$
     for $\theta \in [0,2], $ then 
     \[
     \mathbb{E}\|d_1^n\|^2+\mathbb{E}\|d_2^n\|^2\leq Ck^{\theta +1}(\|A^{\theta /2}Q_1^{1/2}\|_{HS}^2+ \|A^{\theta /2}Q_2^{1/2}\|_{HS}^2).
     \]
     The constant  $C=C(T)$ is independent of $h,k \text{ and } n$ with $t_n=nk \leq T.$
 \end{lemma}     
\begin{proof}
   We consider $d^n_1 \text{ with } \theta \in [0,2].$ By Ito's isometry, \eqref{eqn3.1} and \eqref{eqn3.2} we have
\[
\begin{split}
  \mathbb{E}\|d_1^n\|&=\mathbb{E}\bigg\|\int _{t_n}^{t_{n+1}} (C_h(t_{n+1}-\tau)-C_h(k)) \mathcal{P}_hdW_1(\tau)\\
    &\hspace{3cm}-\int _{t_n}^{t_{n+1}} (S_h(t_{n+1}-\tau)-S_h(k)) \mathcal{P}_hdW_2(\tau)\bigg\|\\
    & \leq 2\mathbb{E}\left\|\int _{t_n}^{t_{n+1}} (C_h(t_{n+1}-\tau)-C_h(k)) \mathcal{P}_hdW_1(\tau)\right\|\\
    & \hspace{0cm}+2\mathbb{E}\left\|\int _{t_n}^{t_{n+1}} (C_h(t_{n+1}-\tau)-C_h(k)) \mathcal{P}_hdW_2(\tau)\right\|\\ 
     &=2\int _{t_n}^{t_{n+1}} \left\|(C_h(t_{n+1}-\tau)-C_h(k)) \mathcal{P}_hQ_1^{1/2}\right\|_{HS} ^2d\tau \\
    &+2\int _{t_n}^{t_{n+1}} \left\|(S_h(t_{n+1}-\tau)-S_h(k)) \mathcal{P}_hQ_2^{1/2}\right\|_{HS} ^2d\tau \\
    &=2\int _{0}^k \left\|(C_h(\tau)-C_h(k)) \mathcal{P}_hQ_1^{1/2}\right\|_{HS} ^2d\tau \\
    &+2\int _{0}^k\left \|(S_h(\tau)-S_h(k)) \mathcal{P}_hQ_2^{1/2}\right\|_{HS} ^2d\tau \\
\end{split}
\]   
\[
\begin{split}
     &=2\int _{0}^k \left\|(C_h(\tau)-C_h(k)) A_h^{-\theta /2}(A_h^{\theta /2}\mathcal{P}_h Q_1^{1/2})\right\|_{HS} ^2d\tau \\
    &+2\int _{0}^k \left\|(S_h(\tau)-S_h(k))A_h^{-\theta /2}(A_h^{\theta /2} \mathcal{P}_hQ_2^{1/2})\right\|_{HS} ^2d\tau \\
&\leq   2\int _{0}^k \left\|(C_h(\tau)-C_h(k)) A_h^{-\theta /2}\right\|^2d\tau \left\|A_h^{\theta /2}\mathcal{P}_h Q_1^{1/2}\right\|_{HS} ^2\\
    &+
  2\int _{0}^k \left\|(S_h(\tau)-S_h(k))A_h^{-\theta /2}\right\|^2d\tau\left\|A_h^{\theta /2} \mathcal{P}_hQ_2^{1/2}\right\|_{HS} ^2 \\
  & \leq C\int _0^k |k-\tau|^{\theta }d\tau \left(\left\|A_h^{\theta /2} \mathcal{P}_hQ_1^{1/2}\right\|_{HS}^2+\left\|A_h^{\theta /2} \mathcal{P}_hQ_2^{1/2}\right\|_{HS}\right)\\
  & \leq C k^{\theta +1 } \left(\left\|A_h^{\theta /2} \mathcal{P}_hQ_1^{1/2}\|_{HS}^2+\|A_h^{\theta /2} \mathcal{P}_hQ_2^{1/2}\right\|_{HS}^2\right).
\end{split}
\]
We use \eqref{eqn2.4} with $\gamma =\theta /2 \in [0,1]$ to get an estimate of  $\left\|A_h^{\theta /2} \mathcal{P}_hQ_j^{1/2}\right\|_{HS}$ for $j=1,2$.
\[
\begin{split}
\left\|A_h^{\theta /2} \mathcal{P}_hQ_j^{1/2}\right\|_{HS}& =\left\|A_h^{\theta /2} \mathcal{P}_h A ^{-\theta /2}(A^{\theta /2 }Q_j^{1/2})\right\|_{HS}\\
&\leq \left\|A_h^{\theta /2} \mathcal{P}_h A ^{-\theta /2}\right\|\left\|A^{\theta /2 }Q_j^{1/2}\right\|_{HS}\\
& \leq C \left\|A^{\theta /2 }Q_j^{1/2}\right\|_{HS}
\end{split}
\]
From this, we have 
     \[
     \mathbb{E}\|d_1^n\|^2\leq Ck^{\theta +1}\left(\left\|A^{\theta /2}Q_1^{1/2}\right\|_{HS}^2+ \left\|A^{\theta /2}Q_2^{1/2}\right\|_{HS}^2\right).
     \]
     Similarly, we get 
   \[
     \mathbb{E}\|d_2^n\|^2\leq Ck^{\theta +1}\left(\left\|A^{\theta /2}Q_1^{1/2}\right\|_{HS}^2+ \left\|A^{\theta /2}Q_2^{1/2}\right\|_{HS}^2\right).
     \]
This completes the proof of lemma.
    \end{proof}
 \subsection{\bf{Mean convergence analysis of linear equation \eqref{eqn1.2}}}
 In this subsection, we will prove mean square error bounds for the stochastic trigonometric method \eqref{eqn1.5}.
 \begin{proof}[\bf Proof of Theorem \ref{thm2}] Define $F^n_j:=U^n_j-u_{h,j}(t_n)$ for $j=1,2$
 and $F^n=(F^n_1,F_2^n)$. From \eqref{Dslinear} we get 
 \[
 X_h(t_{n+1})=e^{k\mathbb{A}_h}X_h(t_n)+e^{k\mathbb{A}_h}\mathcal{P}_h\Delta W^n+d^n
 \]
with $d^n:=(d^n_1,d^n_2)$  defined in lemma \ref{lemma1}. Thus we obtain the following formula for the error $F^{n}$
\[
F^{n}=e^{k\mathbb{A}_h}F^{n-1}+d^{n-1}=e^{t_n\mathbb{A}_h}F^0+\sum_{j=0}e^{t_{n-1-j}\mathbb{A}_h}d^j=\sum_{j=0}e^{t_{n-1-j}\mathbb{A}_h}d^j,
\]
 since  $F^0=0$. We calculate the error estimate for the first component of $F^n$.
 \[
 \begin{split}
 \mathbb{E}\|F^n_1\|_{L_2(\mathcal{O})}^2 &=\mathbb{E}\|\sum _{j=0}^{n-1}C_h(t_{n-1-j})d_1^j-S_h(t_{n-1-j})d_2^j\|\\
 &=\mathbb{E} \Bigg[
  \left(\sum _{j=0}^{n-1}C_h(t_{n-1-j})d_1^j,\sum _{i=0}^{n-1}C_h(t_{n-1-i})d_1^i\right)\\
  & \qquad+ \left(\sum _{j=0}^{n-1}C_h(t_{n-1-j})d_1^j,-\sum _{i=0}^{n-1}S_h(t_{n-1-i})d_2^i\right)\\
  &\qquad+ \left(-\sum _{j=0}^{n-1}S_h(t_{n-1-j})d_2^j,\sum _{i=0}^{n-1}C_h(t_{n-1-i})d_1^i\right)\\
  &\qquad+\left(\sum _{j=0}^{n-1}S_h(t_{n-1-j})d_2^j,\sum _{i=0}^{n-1}S_h(t_{n-1-i})d_2^i\right)\Bigg].
  \end{split}
 \]
Here we use the independence of $d_1^i,d_2^j \text{ with } i,j=0,\cdots ,n-1 \text{ for } i \neq j$ and martingle property to get 
\[
\begin{split}
   \mathbb{E}\|F^n_1\|_{L_2(\mathcal{O})}^2 &=\mathbb{E}\Bigg[\sum_{j=0}^{n-1} \left(C_h(t_{n-1-j})d_1^j,C_h(t_{n-1-j})d_1^j\right)+\sum_{j=0}^{n-1} \left(C_h(t_{n-1-j})d_1^j,-S_h(t_{n-1-j})d_2^j\right)\\
  &+\sum_{j=0}^{n-1} \left(-S_h(t_{n-1-j})d_2^j,C_h(t_{n-1-j})d_1^j\right) + \sum_{j=0}^{n-1} \left(S_h(t_{n-1-j})d_2^j,S_h(t_{n-1-j})d_2^j\right)\Bigg]  \\
  &=\sum_{j=0}^{n-1} \mathbb{E}\|C_h(t_{n-1-j})d_1^j-S_h(t_{n-1-j})d_2^j\|_{L_2(\mathcal{O})}^2\\
  &\leq 2\sum_{j=0}^{n-1}\left(\mathbb{E}\|d_1^j\|_{L_2(\mathcal{O})}^2+E\|d_2^j\|_{L_2(\mathcal{O})}^2\right)
\end{split}
\]
Using the estimates of $d_1^j \text{ and } d_2^j$ from the lemma \ref{lemma1}, we get 
\[
\begin{split}
     \mathbb{E}\|F^n_1\|_{L_2(\mathcal{O})}^2 & \leq 2C \sum_{j=0}^{n-1} k^{\theta +1}(\|A^{\theta /2}Q_1^{1/2}\|_{HS}^2+ \|A^{\theta /2}Q_2^{1/2}\|_{HS}^2)\\
     & \leq 2C k^{\theta +1}n(\|A^{\theta /2}Q_1^{1/2}\|_{HS}^2+ \|A^{\theta /2}Q_2^{1/2}\|_{HS}^2)\\
     & \leq C(T)k^{\theta }(\|A^{\theta /2}Q_1^{1/2}\|_{HS}^2+ \|A^{\theta /2}Q_2^{1/2}\|_{HS}^2).
\end{split}
\]
Therefore we obtain 
\[
\begin{split}
 \|U^n_1-u_{h,1}(t_n)\|_{L_2(\Omega, \dot{H}^0)}&=\left(\mathbb{E}\|F^n_1\|_{L_2(\mathcal{O})}^2 \right)^{1/2}\\
 &\leq \left(C(T)k^{\theta }(\|A^{\theta /2}Q_1^{1/2}\|_{HS}^2+ \|A^{\theta /2}Q_2^{1/2}\|_{HS}^2)\right)^{1/2}\\
 & \leq C(T) k^{\theta/2 }(\|A^{\theta /2}Q_1^{1/2}\|_{HS}+ \|A^{\theta /2}Q_2^{1/2}\|_{HS})
 \end{split}
\]
for $\theta \in [0,2]$. Similarly, we get
\[
 \|U^n_2-u_{h,2}(t_n)\|_{L_2(\Omega, \dot{H}^0)} \leq Ck^{\theta /2}(\|A^{\theta /2}Q_1^{1/2}\|_{HS}+ \|A^{\theta /2}Q_2^{1/2}\|_{HS} )
\]
for $\theta \in [0,2].$
 \end{proof}
 \begin{proof}[\bf Proof of Theorem \ref{thm3}]
This proof follows from Theorem \ref{thm1} and \ref{thm2} by the triangular inequality of norm.     
 \end{proof}
\subsection{Mean convergence analysis of semilinear equation \eqref{eqn1.1}}
In this subsection, we prove the proposition \ref{prop1} and Theorem \ref{thm4}. Our result is a global error estimate for the full discretization in Theorem \ref{thm4}. Its proof is based on bounds and H\"{o}lder regularity of the solution in Proposition \ref{prop1}.
\begin{proof}[\bf Proof of Proposition \ref{prop1}]
   We first find the estimate of the norm $A^{\theta/2}u_1(t)$ and similar way we get estimate of the norm $A^{\theta/2}u_2(t)$. From the first component $u_1(t)$ of \eqref{eqn1.7} we  get
 \[
 \begin{split}
A^{\theta/2} u_1(t)&=A^{\theta/2}(C(t)u_{0,1}-S(t)u_{0,2})\\
&+\int_0^t \left[A^{\theta/2}C(t-\tau)g_1(u_1(\tau),u_2(\tau))-A^{\theta/2}S(t-\tau)g_2(u_1(\tau),u_2(\tau))\right]d\tau\\
& +\int _0^t A^{\theta/2}C(t-\tau)f_1(u_1(\tau),u_2(\tau))dW_1(\tau)\\
&-\int _0^t A^{\theta/2}S(t-\tau)f_2(u_1(\tau),u_2(\tau))dW_2(\tau)\\
 &=I_1+I_2+I_3+I_4,
  \end{split}
 \]
 where 
 \[
 \begin{split}
&I_1:=A^{\theta/2}(C(t)u_{0,1}-S(t)u_{0,2}),\\
& I_2:=\int_0^t \left[A^{\theta/2}C(t-\tau)g_1(u_1(\tau),u_2(\tau))-A^{\theta/2}S(t-\tau)g_2(u_1(\tau),u_2(\tau))\right]d\tau,\\
&I_3:=\int _0^t A^{\theta/2}C(t-\tau)f_1(u_1(\tau),u_2(\tau))dW_1(\tau),\\
&I_4:=-\int _0^t A^{\theta/2}S(t-\tau)f_2(u_1(\tau),u_2(\tau))dW_2(\tau).
 \end{split}
 \]
 Taking the expectation of $\left\|A^{\theta/2}u_1(t)\right\| ^2$ we have 
 \[
 \mathbb{E}\left\|A^{\theta/2}u_1(t)\right\| ^2\leq 4\mathbb{E}\|I_1\|^2+4\mathbb{E}\|I_2\|^2+4\mathbb{E}\|I_3\|^2+4\mathbb{E}\|I_4\|^2.
 \]
 Using the fact that $A$ commutes with the sine and  cosine operators together with a condition on the initial value we have 
 \[
 \begin{split}
     \mathbb{E}\|I_1\|^2&= \mathbb{E}\left\|A^{\theta/2} (C(t)u_{0,1}-S(t)u_{0,2})\right\|^2\\
     & \leq 2\mathbb{E} \left\|A^{\theta/2} C(t)u_{0,1}\right\|^2+2\mathbb{E} \left\|A^{\theta/2} S(t)u_{0,2}\right\|^2\\
     & \leq 2 \mathbb{E}\left\|A^{\theta /2}u_{0,1}\right\|^2+2 \mathbb{E}\left\|A^{\theta /2}u_{0,2}\right\|^2\\
     & \leq C_1\mathbb{E}|||X_0|||_{\theta}.
     \end{split}
 \]

\noindent To estimate the second term, we get
 \[
 \begin{split}
 \mathbb{E}\|I_2\|^2 &=\mathbb{E}\left\|\int _0^t A ^{\theta/2} C(t-\tau)g_1(u_1(\tau),\,u_2(\tau))d\tau-\int _0^t A ^{\theta/2} S(t-\tau)g_2(u_1(\tau),\,u_2(\tau))d\tau\right\|^2\\
 &\leq 2t\mathbb{E} \int_0^t\left \|A^{\theta /2}C(t-\tau)g_1(u_1(\tau),\, u_2(\tau))\right\|^2d\tau\\
 &+2t\mathbb{E}  \int_0^t \left\|A^{\theta /2}C(t-\tau)g_1(u_1(\tau),\, u_2(\tau))\right\|^2d\tau\\
 & \leq 2t \int_0^t \mathbb{E}\left[\left\|A ^{\theta/2}g_1(u_1(\tau),\,u_2(\tau))\right\|^2+\left\|A ^{\theta/2}g_1(u_1(\tau),\,u_2(\tau))\right\|^2\right]d\tau\\
 & \leq 4t L_g \int_0^t \mathbb{E}\left[\left\|A^{\theta/2}u_1(\tau)\|^2+ \|A^{\theta/2}u_2(\tau)\right\|^2+1 \right]d\tau.
 \end{split}
 \] 
 Using It\^{o} isometry to estimate the third term, we have 
 \[
 \begin{split}
 \mathbb{E}\|I_3\|^2 &=\mathbb{E}\left\|\int_0^tA^{\theta /2}C(t-\tau)f_1(u_1(\tau),u_2(\tau))dW_1(\tau)\right\|^2\\
 & =\mathbb{E}\int_0^t\left\|A^{\theta /2}C(t-\tau)f_1(u_1(\tau),u_2(\tau))Q^{1/2}_1\right\|^2 _{HS}d\tau\\
 &\leq L_f\int_0^t \mathbb{E} \left(\left\|u_1(\tau)\right\|_{\theta}^2+\left\|u_2(\tau)\right\|_{\theta }^2+1\right)d\tau.
 \end{split}
 \]
 Similar to the above estimate we have 
 \[
 \begin{split}
 \mathbb{E}\|I_4\|^2&=\mathbb{E}\left\|\int_0^t A^{\theta/2}S(t-\tau)f_2(u_1(\tau),u_2(\tau))dW_2(\tau)\right\|^2 \\
 &\leq L_f\int_0^t \mathbb{E} \left(\left\|u_1(\tau)\right\|_{\theta}^2+\left\|u_2(\tau)\right|_{\theta }^2+1\right)d\tau.
 \end{split}
 \]
 We collect all of the above estimates, and we see
 \[
 \begin{split}
 \mathbb{E}\left\|A^{\theta/2}u_1(t)\right\|^2&\leq C_1\mathbb{E}|||X_0|||_{\theta}+\left(4tL_g+2L_f\right) \int_0^t \mathbb{E}\left[\left\|u_1(\tau)\right\|^2_{\theta}+ \left\|u_2(\tau)\right\|^2_{\theta} +1\right]d\tau\\
 &\leq C_1\mathbb{E}|||X_0|||_{\theta}+tC_2+C_2\int_0^t \mathbb{E}\left[\left\|u_1(\tau)\right\|^2_{\theta}+ \left\|u_2(\tau)\right\|^2_{\theta} \right]d\tau
 \end{split}
 \]
 Similarly, we get 
  \[
 \mathbb{E}\left\|A^{\theta/2}u_2(t)\right\|^2\leq C_3\mathbb{E}|||X_0|||_{\theta}+C_4 t+C_4 \int_0^t \mathbb{E}\left[\left\|u_1(\tau)\right\|^2_{\theta}+ \left\|u_2(\tau)\right\|^2_{\theta} \right]d\tau.
 \]
 We sum the above estimates and an application of Gronwall's lemma, we see
 \[
 \begin{split}
  \mathbb{E}\left[\left\|u_1(t)\right\|^2_{\theta}+\left\|u_2(t)\right\|^2_{\theta}\right] &\leq \tilde{C_1}+\tilde{C_2}\int_0^t \mathbb{E}\left[\left\|u_1(\tau)\right\|^2_{\theta}+ \left\|u_2(\tau)\right\|^2_{\theta} \right]d\tau\\
  & \leq \tilde{C_1}e^{\tilde{C_2}t}
  \end{split}
 \]
 Taking the supremum on $[0,T]$, we get 
 \[
 \sup_{0 \leq t \leq T}  \mathbb{E}\left[\left\|u_1(t)\right\|^2_{\theta}+\left\|u_2(t)\right\|^2_{\theta}\right]\leq C.
 \]
We now prove an H\"{o}lder regularity property of the solution. We first find an H\"{o}lder regularity property of the first component $u_1(t)$ of \eqref{eqn1.7} and in a similar way we get an H\"{o}lder regularity property of the second component $u_2(t)$. We  write, for $0 \leq s\leq t \leq T$,
\[
\begin{split}
    u_1(t)-u_1(s)&=(C(t)-C(s))u_{0,1}-(S(t)-S(s))u_{0,2}\\
    &+\int_0^s\left(C(t-r)-C(s-r)\right)g_1(u_1(r),u_2(r))dr\\
    &+ \int_s^tC(t-r) g_1(u_1(r),u_2(r))dr-\int_0^s(S(t-r)-S(s-r))g_2(u_1(r),u_2(r))dr\\
    &- \int_s^tS(t-r) g_2(u_1(r),u_2(r))dr\\
    &+ \int_0^s (C(t-r)-C(s-r))f_1(u_1(s),u_2(s))dW_1(r)\\
    &+\int_s^t C(t-r)f_1(u_1(s),u_2(s))dW_1(r)
    \\
    &-\int_0^s (S(t-r)-S(s-r))f_2(u_1(s),u_2(s))dW_2(r)\\
    &-\int_s^t S(t-r)f_2(u_1(s),u_2(s))dW_2(r)\\
    &:=J_1+J_2+J_3+J_4+J_5+J_6+J_7+J_8+J_9+J_{10}.
\end{split}
\]
Taking the expectation of $\|u_1(t)-u_1(s)\|^2$, we get
\[
\begin{split}
\mathbb{E}\|u_1(t)-u_1(s)\|^2 & \leq 10\mathbb{E}\|J_1\|^2+10\mathbb{E}\|J_2\|^2+10\mathbb{E}\|J_3\|^2+10\mathbb{E}\|J_4\|^2+10\mathbb{E}\|J_5\|^2\\&+10
\mathbb{E}\|J_6\|^2+10\mathbb{E}\|J_7\|^2
+10\mathbb{E}\|J_8\|^2+10\mathbb{E}\|J_9\|^2
+10\mathbb{E}\|J_{10}\|^2.
\end{split}
\]

To estimate the first term, we use \eqref{eqn3.3} to get
\[
\begin{split}
\mathbb{E}\|J_1\|^2&=\mathbb{E}\left\|(C(t)-C(s))u_{0,1}\right\|^2=\mathbb{E}\left\|(C(t)-C(s))A^{-\theta/2}(A^{\theta /2}u_{0,1})\right\|^2\\
&\leq \left\|(C(t)-C(s))A^{-\theta/2}\right\|^2\left\|A^{\theta /2}u_{0,1}\right\|^2 \leq C|t-s|^{\theta}\mathbb{E}\left\|u_{0,1}\right\|^2_{\theta}.
\end{split}
\]
Similar to the above we have 
\[
\mathbb{E}\|J_2\|^2=\mathbb{E}\left\|(S(t)-S(s))u_{0,2}\right\|^2\leq C|t-s|^{\theta}\mathbb{E}\|u_{0,2}\|^2_{\theta}.
\]
To estimate the third term, we use \eqref{eqn3.3} to get 
\[
\begin{split}
    \mathbb{E}\|J_3\|^2&=\mathbb{E}\left\|\int_0^s(C(t-r)-C(s-r))g_1(u_1(r),\,u_2(r))dr\right\|^2\\
    & \leq s \int_0^s \mathbb{E}\left \|(C(t-r)-C(s-r))g_1(u_1(r),\,u_2(r))\right\|^2dr\\
    &\leq T \int_0^s \mathbb{E}\left \|(C(t-r)-C(s-r))A^{-\theta/2}(A^{\theta/2}g_1(u_1(r),\,u_2(r)))\right\|^2dr\\
    & \leq C|t-s|^{\theta } \int_0^s \mathbb{E}\left\|A^{\theta/2}g_1(u_1(r),\,u_2(r))\right\|^2dr\\
    &\leq CL_g|t-s|^{\theta} \sup_{0 \leq t \leq T}\mathbb{E}\left[\left\|u_1(t)\right\|^2_{\theta}+\left\|u_2(t)\right\|^2_{\theta}+1\right].
\end{split}
\]
To estimate the fourth term, we use the operator 
 norm bound of the cosine operator to get 
\[
\begin{split}
    \mathbb{E}\|J_4\|^2&=\mathbb{E}\left\|\int_s^t C(t-r)g_1(u_1(r),\,u_2(r))dr\right\|^2\\
    & \leq(t-s)\int_s^t \mathbb{E}\left\|C(t-r)g_1(u_1(r),\,u_2(t))\right\|^2dr\\
    & \leq T  \int_s^t \mathbb{E}\left\|g_1(u_1(r),\,u_2(r))\right\|^2dr\\
    &\leq T L_g\int_s^t \mathbb{E}\left[\left\|u_1(r)\right\|^2+\left\|u_2(r)\right\|^2+1\right]dr\\
    &\leq CL_g|t-s| \sup_{0 \leq t \leq T}\mathbb{E}\left[\left\|u_1(t)\right\|^2_{\theta}+\left\|u_2(t)\right\|^2_{\theta}+1\right].
\end{split}
\]
Similar to the third estimate we have 
\[
\begin{split}
\mathbb{E} \|J_5\|^2&=\mathbb{E}\left\|\int_0^s S(t-r)-S(s-r)g_2(u_1(r),u_2(r))dr\right\|^2\\
& \leq  CL_g|t-s|^{\theta} \sup_{0 \leq t \leq T}\mathbb{E}\left[\left\|u_1(t)\right\|^2_{\theta}+\left\|u_2(t)\right\|^2_{\theta}+1\right].
\end{split}
\]
For an estimate of the sixth  term, we use \eqref{eqn3.3} to get
\[
\begin{split}
\mathbb{E}\|J_6\|^2&=\mathbb{E}\left\|\int_s^tS(t-r)g_2(u_1(r),u_2(r))dr\right\|^2\\
&\leq C\int_s^t \mathbb{E} \left\|S(t-r)A^{-\theta/2}(A^{\theta/2}g_2(u_1(r),u_2(r)))\right\|^2dr\\
&\leq C \int_s^t \left\|S(t-r)A^{-\theta/2}\right\|^2\mathbb{E}\left\|A^{\theta/2}g_2(u_1(r),u_2(r))\right\|^2dr\\
& \leq CL_g|t-s|^{\theta} \sup_{0 \leq t \leq T}\mathbb{E}\left[\left\|u_1(t)\right\|^2_{\theta}+\left\|u_2(t)\right\|^2_{\theta}+1\right].
\end{split}
\]
To estimate the seventh term, we It\^{o} isometry and \eqref{eqn3.3} to get 
\[
\begin{split}
\mathbb{E}\|J_7\|^2&=\mathbb{E}\left\|\int_0^s(C(t-r)-C(s-r))f_1(u_1(r),u_2(r))dW_1(r)\right\|^2\\
& =\mathbb{E}\int _0^s\left\|(C(t-r)-C(s-r))f_1(u_1(r),u_2(r))Q_1^{1/2}\right\|_{HS}^2dr\\
& \leq C|t-s|^{\theta }\int_0^t \mathbb{E}\left\|A^{\theta/2}f_1(u_1(r),u_2(r))Q_1^{1/2}\right\|dr\\
& \leq CL_f|t-s|^{\theta}\sup_{0 \leq t \leq T}\left(\|u_1(t)\|_{\theta}^2+\|u_2(t)\|_{\theta}^2+1\right).
\end{split}
\]
We use the operator norm bound of the cosine operator and It\^{o} isometry to get 
\[
\begin{split}
\mathbb{E}\|J_8\|^2&=\mathbb{E}\left\|\int_s^tC(t-r)f_1(u_1(r),u_2(r))Q_1^{1/2}dW_1(r)\right\|^2\\
&=\mathbb{E}\int_s^t\left\|C(t-r)f_1(u_1(r),u_2(r))Q^{1/2}_1\right\|^2_{HS}\\
& \leq \int_s^t \mathbb{E}\left\|A^{\theta/2}f_1(u_1(r),u_2(r))Q^{1/2}_1\right\|^2_{HS}\\
& \leq |t-s|L_f\sup_{0 \leq t \leq T}\left(\|u_1(t)\|_{\theta}^2+\|u_2(t)\|_{\theta}^2+1\right).
\end{split}
\]
Similar to the seventh, we get 
\[
\begin{split}
 \mathbb{E}\|J_9\|^2&=\mathbb{E}\left\|\int_0^s(S(t-r)-s(s-r))f_2(u_1(r),u_2(r))dW_2(r)\right\|^2\\ 
 & \leq CL_f|t-s|^{\theta}\sup_{0 \leq t \leq T}\left(\|u_1(t)\|_{\theta}^2+\|u_2(t)\|_{\theta}^2+1\right).
\end{split}
\]
For the last estimate, we have 
\[
\begin{split}
\mathbb{E}\|J_{10}\|^2&=\mathbb{E}\left\|\int_s^t S(t-r)f_2(u_1(r),u_2(r))dW_2(r)\right\|^2\\
&=\mathbb{E}\int_s^t \left\|S(t-r)f_2(u_1(r),u_2(r))Q_2^{1/2}\right\|^2_{HS}ds\\
&\leq |t-s|^{\theta}L_g\sup_{0 \leq t \leq T}(\|u_1(t)\|_{\theta}^2+\|u_2(t)\|_{\theta}^2+1).
\end{split}
\]
Collecting the above estimates, we get 
\[
\begin{split}
    \mathbb{E}\|u_1(t)-u_1(s)\|^2&\leq C_1|t-s|^{\theta}\left(\left\|u_{0,1}\right\|^2_{\theta}+\left\|u_{0,2}\right\|^2_{\theta}+\sup_{0 \leq t \leq T}\mathbb{E}\left(\|u_1(t)\|^2_{\theta}+\|u_2(t)\|^2_{\theta}+1\right)\right)\\
    &+C_2|t-s|\left(\sup_{0 \leq t \leq T}\mathbb{E}[\|u_1(t)\|^2_{\theta}+\|u_2(t)\|^2_{\theta}+1]\right).
\end{split}
\] 
This finishes the proof.
\end{proof}
\begin{proof}[\bf Proof of Theorem \ref{thm4}]
 We see the error $U^n-X(t_n)$ between the numerical and the exact solutions where $U^n$ is given in \eqref{rformula} and $X(t_n)$ is obtained from \eqref{eqn1.7}. The error estimate is provided by 
\[
\begin{split}
\mathbb{E}|||U^n-X(t_n)|||^2&=\mathbb{E}\bigg|\bigg|\bigg|\left(e^{t_{n}\mathbb{A}_h}U^0-e^{t_{n}\mathbb{A}}X_0\right)\\
&\qquad+\left(\sum_{j=0}^{n-1}e^{(t_{n}-t_j)\mathbb{A}_h}\mathcal{P}_hG(U^j)k-\int _0^{t_n}e^{(t_{n}-\tau)\mathbb{A}}G(X_\tau)d\tau\right)\\
&\qquad +\left(\sum_{j=0}^{n-1}e^{(t_{n}-t_j)\mathbb{A}_h}\mathcal{P}_hF(U^j)\Delta W^j-\int_0^{t_n}e^{(t_{n}-\tau)\mathbb{A}}F(X_{\tau})dW(\tau)\right)\bigg|\bigg|\bigg|^2\\
&\leq 3\mathbb{E}|||e^{t_{n}\mathbb{A}_h}U^0-e^{t_{n}\mathbb{A}}X_0|||^2\\
\end{split}
\]
\[
\begin{split}
~~~~~~~~~~~~&+3\mathbb{E}\bigg|\bigg|\bigg|\sum_{j=0}^{n-1}\int _{t_j}^{t_{j+1}}\left(e^{(t_{n}-t_j)\mathbb{A}_h}\mathcal{P}_hG(U^j)-e^{(t_{n}-\tau)\mathbb{A}}G(X_\tau)\right)d\tau\bigg|\bigg|\bigg|^2\\
&+3\mathbb{E}\bigg|\bigg|\bigg|\sum_{j=0}^{n-1}\int _{t_j}^{t_{j+1}}\left(e^{(t_{n}-t_j)\mathbb{A}_h}\mathcal{P}_hF(U^j)-e^{(t_{n}-\tau)\mathbb{A}})F(X(\tau)\right)dW(\tau)\bigg|\bigg|\bigg|^2\\
&:=3\mathbb{E}|||Err_0|||^2+3\mathbb{E}|||Err_d|||^2+3\mathbb{E}|||Err_s|||^2.
\end{split}
\]
We now estimate the above three errors.

\textbf{Estimate for the initial error $Err_0$}. We use Theorem $1.3$ of \cite{finite}.
Error estimate in the first component of $Err_0$ is given by 
\[
\begin{split}
\mathbb{E}\|Err_{0,1}\|^2&=\mathbb{E}\|(C_h(t_n)\mathcal{P}_h-C(t_n))u_{0,1}-(S_h(t_n)\mathcal{P}_h-S(t_n))u_{0,2}\|^2\\
& \leq C^2h^{2\theta}\mathbb{E}(\|u_{0,1}\|^2_{\theta}+\|u_{0,2}\|^2_{\theta})
\end{split}
\]
for $\theta \in [0,2]$.
Similarly, for the second component
\[
\mathbb{E}\|Err_{0,2}\|^2 \leq  C^2h^{2\theta}\mathbb{E}(\|u_{0,1}\|^2_{\theta}+\|u_{0,2}\|^2_{\theta})
\]
for $\theta \in [0,2]$.

\textbf{Estimate for the deterministic error $Err_d.$} We write the deterministic error $Err_d$ as 
\[
\begin{split}
    Err_d&=\sum_{j=0}^{n-1}\int _{t_j}^{t_{j+1}}(e^{(t_n-t_j)\mathbb{A}_h}\mathcal{P}_hG(U^j)-e^{(t_n-\tau)\mathbb{A}}G(X(\tau)))d\tau\\
    &=\sum_{j=0}^{n-1} \int_{t_j}^{t_{j+1}}e^{(t_n-t_j)\mathbb{A}_h}\mathcal{P}_h(G(U^j)-G(X(t_j)))d\tau\\
    &+\sum_{j=0}^{n-1} \int_{t_j}^{t_{j+1}}e^{(t_n-t_j)\mathbb{A}_h}\mathcal{P}_h(G(X(t_j))-G(X(\tau)))d\tau\\
    &+\sum_{j=0}^{n-1} \int_{t_j}^{t_{j+1}}(e^{(t_n-t_j)\mathbb{A}_h}\mathcal{P}_h-e^{(t_n-t_j)\mathbb{A}})G(X(\tau))d\tau\\
    &+\sum_{j=0}^{n-1} \int_{t_j}^{t_{j+1}}(e^{(t_n-t_j)\mathbb{A}}-e^{(t_n-\tau)\mathbb{A}})G(X(\tau))d\tau\\
    &:=\tilde{I}_1+\tilde{I}_2+\tilde{I}_3+\tilde{I}_4
\end{split}
\]
We next estimate the second moment of each term in the above equation. In the first component of the first term $\tilde{I}_1$, we see
\[
\begin{split}
   \left( \mathbb{E}\|\tilde{I}_{1,1}\|^2\right)^{1/2}&=\bigg(\mathbb{E}\bigg\|\sum_{j=0}^{n-1}\int _{t_j}^{t_{j+1}}\bigg[C_h(t_n-t_j)\mathcal{P}_h(g_1(U_1^j,\,U_2^j)-g_1(u_1(t_j),\,u_2(t_j)))\\ &\hspace{3cm}-S_h(t_n-t_j)\mathcal{P}_h(g_2(U_1^j,\,U_2^j)
   -g_2(u_1(t_j),\,u_2(t_j)))\bigg]d\tau \bigg\|^2\bigg)^{1/2}\\
   & \leq \sum_{j=0}^{n-1} \int_{t_j}^{t_{j+1}}\left(\mathbb{E}\|C_h(t_n-t_j)\mathcal{P}_h(g_1(U_1^j,\,U_2^j)-g_1(u_1(t_j),\,u_2(t_j)))\|^2\right)^{1/2}d\tau\\
\end{split}
\]
\[
\begin{split}
~~~~~~~~~~~~~~~~~~~~~~~~~~~&+\sum_{j=0}^{n-1} \int_{t_j}^{t_{j+1}}\left(\mathbb{E}\|S_h(t_n-t_j)\mathcal{P}_h(g_2(U_1^j,\,U_2^j)-g_2(u_1(t_j),\,u_2(t_j)))\|^2\right)^{1/2}d\tau\\
   & \leq k\sum_{j=0}^{n-1} \bigg[\left(\mathbb{E}\|g_1(U_1^j,\,U_2^j)-g_1(u_1(t_j),\,u_2(t_j))\|^2\right)^{1/2}\\
   & \hspace{3.4cm}+\left(\mathbb{E}\|g_2(U_1^j,\,U_2^j)-g_2(u_1(t_j),\,u_2(t_j))\|^2\right)^{1/2}\bigg]\\
   & \leq \sqrt{2}k \sum_{j=0}^{n-1}\bigg(\mathbb{E}\|g_1(U_1^j,\,U_2^j)-g_1(u_1(t_j),\,u_2(t_j))\|^2\\
   & \hspace{3.4cm}+\mathbb{E}\|g_2(U_1^j,\,U_2^j)-g_2(u_1(t_j),\,u_2(t_j))\|^2\bigg)^{1/2}
\end{split}
\]
so that
\[
\begin{split}
     \mathbb{E}\left\|\tilde{I}_{1,1}\right\|^2& \leq 2k^2n \sum_{j=0}^{n-1}\bigg(\mathbb{E}\left\|g_1(U_1^j,\,U_2^j)-g_1(u_1(t_j),\,u_2(t_j))\right\|^2\\
     &\hspace{4cm}+\mathbb{E}\left\|g_2(U_1^j,\,U_2^j)-g_2(u_1(t_j),\,u_2(t_j))\right\|^2\bigg)\\
     & \leq Ck\sum_{j=0}^{n-1}\bigg(\mathbb{E}\left\|g_1(U_1^j,\,U_2^j)-g_1(u_1(t_j),\,u_2(t_j))\right\|^2\\
     &\hspace{4cm}+\mathbb{E}\left\|g_2(U_1^j,\,U_2^j)-g_2(u_1(t_j),\,u_2(t_j))\right\|^2\bigg)\\
     & \leq CK_gk\sum_{j=0}^{n-1}\mathbb{E}\left(\left\|U_1^j-u_1(t_j)\right\|^2+\left\|U_2^j-u_2(t_j)\right\|^2\right).
\end{split}
\]
similarly for the second component we have 
\[
\mathbb{E}\|\tilde{I}_{1,2}\|^2\leq CK_gk\sum_{j=0}^{n-1}\mathbb{E}\left(\left\|U_1^j-u_1(t_j)\right\|^2+\left\|U_2^j-u_2(t_j)\right\|^2\right).
\]
  For the first component of the second term $\tilde{I}_2$, we have 
  \[
  \begin{split}
\left(\mathbb{E}\left\|\tilde{I}_{2,1}\right\|^2\right)^{1/2}&=\bigg(\mathbb{E}\bigg\|\sum_{j=0}^{n-1}\int_{t_j}^{t_{j+1}}\bigg[C_h(t_n-t_j)\mathcal{P}_h\left(g_1(u_1(t_j),\, u(t_j))-g_1(u_1(\tau),u_2(\tau))\right)\\
    & \hspace{2cm}-S_h(t_n-t_j)\mathcal{P}_h(g_2(u_1(t_j),\, u(t_j))-g_2(u_1(\tau),u_2(\tau)))\bigg]\bigg\|^2\bigg)^{1/2}d\tau\\
    & \leq \sum_{j=0}^{n-1} \int_{t_j}^{t_{j+1}}\left(\mathbb{E}\left\|g_1(u_1(t_j),\, u_2(t_j))-g_1(u_1(\tau),\, u_2(\tau))\right\|^2\right)^{1/2}d\tau\\
    & +\sum_{j=0}^{n-1} \int_{t_j}^{t_{j+1}}\left(\mathbb{E}\left\|g_2(u_1(t_j),\, u_2(t_j))-g_2(u_1(\tau),\, u_2(\tau))\right\|^2\right)^{1/2}d\tau\\  
  \end{split}
  \]
\[
\begin{split}
    &\leq \sqrt{2}K_g \sum_{j=0}^{n-1} \int_{t_j}^{t_{j+1}}\left(\mathbb{E}\left[\|u_1(t_j)-u_1(\tau)\|^2+\|u_2(t_j)-u_2(\tau)\|^2\right]\right)^{1/2}d\tau\\
    & \leq 2K_gK_h\sum_{j=0}^{n-1}\int_{t_j}^{t_{j+1}}|t_j-\tau|^{\min(\theta/2,1/2)} d\tau\\
    & \leq Ck^{\min (\theta/1,1/2)}
    \end{split}
\]
for $\theta \in [0,2]$ and $K_h$ is a H\"{o}lder constant. Above we have used the H\"{o}lder regularity of $u_1 \text{ and } u_2$ from Proposition \ref{prop1} and the value of $k$ less than one. Thus we have 
\[
\mathbb{E}\|\tilde{I}_{2,1}\|^2\leq Ck^{\min (\theta ,1)}.
\] 
Similarly, for the second component, we get
\[
\mathbb{E}\|\tilde{I}_{2,2}\|^2\leq C k^{\min(\theta,1)}.
\]
For the first component of the third term $\tilde{I}_3$ we use Theorem $1.3$ of \cite{finite} to get 
\[
\begin{split}
    \left(\mathbb{E}\left\|\tilde{I}_{3,1}\right\|^2\right)^{1/2}&\leq\sum_{j=0}^{n-1} \int_{t_j}^{t_{j+1}}\left(\mathbb{E}\left\|(C_h(t_n-t_j)\mathcal{P}_h-C(t_n-t_j))g_1(u_1(\tau),u_2(\tau))\right\|^2\right)^{1/2}d\tau\\
    &+\sum_{j=0}^{n-1} \int_{t_j}^{t_{j+1}}\left(\mathbb{E}\|(S_h(t_n-t_j)\mathcal{P}_h-S(t_n-t_j))g_2(u_1(\tau),u_2(\tau))\|^2\right)^{1/2}d\tau\\
    &\leq Ch^{\theta}\sum_{j=0}^{n-1} \int_{t_j}^{t_{j+1}}\left[\left(\mathbb{E}\|g_1(u_1(\tau),\,u_2(\tau))\|_{\theta}^2\right)^{1/2}+(\mathbb{E}\|g_2(u_1(\tau),\,u_2(\tau))\|_{\theta}^2)^{1/2}\right]d\tau\\
    & \leq Ch^{\theta}\sum_{j=0}^{n-1}\int_{t_j}^{t_{j+1}} \left[\mathbb{E}\|g_1(u_1(\tau),\,u_2(\tau))\|_{\theta}^2+\mathbb{E}\|g_2(u_1(\tau),\,u_2(\tau))\|_{\theta}^2\right]^{1/2}d\tau\\
    & \leq CL_gh^{\theta} \sum_{j=0}^{n-1} \int_{t_j}^{t_{j+1}}\left(\mathbb{E}(\|u_1(\tau)\|_{\theta}^2+\|u_2(\tau)\|_{\theta}+1\right)^{1/2}d\tau\\
    & \leq CL_gh^{\theta}(nk)\left(\sup_{0 \leq t \leq T}\mathbb{E}(\|u_1(t)\|^2_{\theta}+\|u_1(t)\|^2_{\theta})+1\right)^{1/2}
\end{split}
\]
\noindent for $\theta \text{ in } [0,2]$. Thus,
\[
\mathbb{E}\|\tilde{I}_{3,1}\|^2 \leq Ch^{2 \theta}\sup_{0 \leq t \leq T} \mathbb{E}\left(\|u_1(t)\|^2_{\theta}+\|u_2(t)\|^2_{\theta}+1\right).
\]
\noindent Similarly, for the second component, we get
\[
\mathbb{E}\|\tilde{I}_{3,2}\|^2 \leq Ch^{2 \theta}\sup_{0 \leq t \leq T} \mathbb{E}(\|u_1(t)\|^2_{\theta}+\|u_2(t)\|^2_{\theta}+1).
\]
For the first component of the last term $\tilde{I}_4$, we use \eqref{eqn3.3} to get 
\[
\begin{split}
    \left(\mathbb{E}\left\|\tilde{I}_{4,1}\right\|^2\right)^{1/2}& \leq \sum_{j=0}^{n-1}\int_{t_j}^{t_{j+1}} \left(\mathbb{E}\left\|(C(t_n-t_j)-C(t_n-\tau))g_1(u_1(\tau),u_2(\tau))\right\|^2\right)^{1/2}d\tau\\
    &+\sum_{j=0}^{n-1}\int_{t_j}^{t_{j+1}} \left(\mathbb{E}\left\|(S(t_n-t_j)-S(t_n-\tau))g_2(u_1(\tau),u_2(\tau))\right\|^2\right)^{1/2}d\tau\\
    & \leq \sum_{j=0}^{n-1}\int_{t_j}^{t_{j+1}}C|\tau-t_j|^{\theta/2}\bigg[\left(\mathbb{E}\|A^{\theta/2}g_1(u_1(\tau),\, u_2(\tau))\|^2\right)^{1/2}\\
    & \hspace{3.2cm}+\left(\mathbb{E}\|A^{\theta/2}g_2(u_1(\tau),\, u_2(\tau))\|^2\right)^{1/2}\bigg]d\tau\\
    & \leq CL_gk^{\theta/2} \left(\sup_{0 \leq t \leq T}\mathbb{E}\left(\|u_1(t)\|^2_{\theta}+\|u_2(t)\|^2_{\theta}+1\right)\right)^{1/2}
\end{split}
\]
for $\theta \in [0,2]$. Thus,
\[
\mathbb{E}\|\tilde{I}_{4,1}\|^2\leq Ck^{\theta}\sup_{0 \leq t \leq T}\mathbb{E}(\|u_1(t)\|^2_{\theta}+\|u_2(t)\|^2_{\theta}+1).
\]
Similarly, for the second component we have 
\[
\mathbb{E}\|\tilde{I}_{4,2}\|^2 \leq Ck^{\theta} \sup_{0 \leq t \leq T}\mathbb{E}(\|u_1(t)\|^2_{\theta}+\|u_2(t)\|^2_{\theta}+1).
\]
By consolidating all the preceding estimates for the deterministic error, we obtain the expectation of the norm bound for the first component of $Err_d$,  
 given by:
\[
\mathbb{E}\|Err_{d,1}\|^2\leq C\left[h^{2\theta}+k^{min(\theta ,1)}+k \sum_{j=0}^{n-1}\mathbb{E}\left(\left\|U_1^j-u_1(t_j)\right\|^2+ \left\|U_2^j-u_2(t_j)\right\|^2\right)\right] \text{ for } \theta\in [0,2].
\]
Here, we have used the fact that $k < 1$ and applied the upper bound property from Proposition \ref{prop1}.

Similarly, the expectation of the norm bound for the second component of $Err_d$ is expressed as:
\[
\mathbb{E}\|Err_{d,2}\|^2\leq C\left[h^{2\theta}+k^{min(\theta ,1)}+k \sum_{j=0}^{n-1}\mathbb{E}\left(\left\|U_1^j-u_1(t_j)\right\|^2+ \left\|U_2^j-u_2(t_j)\right\|^2\right)\right]  \text{ for } \theta\in [0,2].
\]
\textbf{Estimate for the stochastic error $Err_s$.} We write stochastic error as 
\[
\begin{split}
Err_s &=\sum_{j=0}^{n-1} \int_{t_j}^{t_{j+1}}\left(e^{(t_n-t_j)\mathbb{A}_h}\mathcal{P}_hF(U^j)-e^{(t_n-\tau)\mathbb{A}}\right)F(X(\tau))dW(\tau)\\
&=\sum_{j=0}^{n-1} \int_{t_j}^{t_{j+1}}e^{(t_n-t_j)\mathbb{A}_h}\mathcal{P}_h(F(U^j)-F(X(t_j)))dW(\tau)\\
&+\sum_{j=0}^{n-1} \int_{t_j}^{t_{j+1}}e^{(t_n-t_j)\mathbb{A}_h}\mathcal{P}_h(F(X(t_j))-F(X(\tau)))dW(\tau)
\\&+\sum_{j=0}^{n-1}\int_{t_j}^{t_{j+1}}(e^{(t_n-t_j)\mathbb{A}_h}\mathcal{P}_h-e^{(t_n-t_j)\mathbb{A}})F(X(\tau))dW(\tau)\\
&+\sum_{j=0}^{n-1}\int_{t_j}^{t_{j+1}}(e^{(t_n-t_j)\mathbb{A}}\mathcal{P}_h-e^{(t_n-\tau)\mathbb{A}})F(X(\tau))dW(\tau)\\
&:=\tilde{J}_1+\tilde{J}_2+\tilde{J}_3+\tilde{J}_4.
\end{split}
\]
For the first component of the first term 
$\tilde{J}_1$, we apply the Itô isometry, the boundedness of 
$C_h$
  and 
$\mathcal{P}_h$
 , and the conditions of the functions 
$f_1$
  and 
$f_2$
  given in \eqref{assum} to obtain:
 \[
 \begin{split}
     \mathbb{E}\|\tilde{J}_{1,1}\|^{2}&=  \mathbb{E}\bigg\|\sum_{j=0}^{n-1}\int_{t_j}^{t_{j+1}}C_h(t_n-t_j)\mathcal{P}_h(f_1(U^j_1,U^j_2)-f_1(u_1(t_j),u_2(t_j)))dW_1(\tau)\\
     &\qquad-\sum_{j=0}^{n-1}\int_{t_j}^{t_{j+1}}S_h(t_n-t_j)\mathcal{P}_h(f_2(U^j_1,U^j_2)-f_2(u_1(t_j),u_2(t_j)))dW_2(\tau)\bigg\|^2\\
     &\leq 2\mathbb{E}\left\|\sum_{j=0}^{n-1}\int_{t_j}^{t_{j+1}}C_h(t_n-t_j)\mathcal{P}_h(f_1(U^j_1,U^j_2)-f_1(u_1(t_j),u_2(t_j)))dW_1(\tau)\right\|^2\\
     &+2\mathbb{E}\left\|\sum_{j=0}^{n-1}\int_{t_j}^{t_{j+1}}S_h(t_n-t_j)\mathcal{P}_h(f_2(U^j_1,U^j_2)-f_2(u_1(t_j),u_2(t_j)))dW_2(\tau)\right\|^2\\
    & =2\mathbb{E}\sum_{j=0}^{n-1}\int_{t_j}^{t_{j+1}}\left\|C_h(t_n-t_j)\mathcal{P}_h(f_1(U^j_1,U^j_2)-f_1(u_1(t_j),u_2(t_j)))Q^{1/2}_1\right\|^2_{HS}d(\tau)\\
    &+2\mathbb{E}\sum_{j=0}^{n-1}\int_{t_j}^{t_{j+1}}\left\|S_h(t_n-t_j)\mathcal{P}_h(f_2(U^j_1,U^j_2)-f_2(u_1(t_j),u_2(t_j)))Q^{1/2}_2\right\|^2_{HS}d(\tau)\\
    &\leq 2k\sum_{j=0}^{n-1}\mathbb{E}\left\|(f_1(U^j_1,U^j_2)-f_1(u_1(t_j),u_2(t_j))Q^{1/2}_1\right\|^2_{HS}d\tau\\
    &+2k\sum_{j=0}^{n-1}\mathbb{E}\left\|(f_1(U^j_1,U^j_2)-f_2(u_1(t_j),u_2(t_j))Q^{1/2}_2\right\|^2_{HS}d\tau\\
    &\leq kC\sum_{j=0}^{n-1}\mathbb{E}\left(\left\|U^j_1-u_1(t_j)\right\|^2+\left\|U^j_2-u_2(t_j)\right\|^2\right)
 \end{split}
 \]
Similarly for the second component we have 
\[
 \mathbb{E}\|\tilde{J}_{1,2}\|^2\leq kC\sum_{j=0}^{n-1}\mathbb{E}\left(\left\|U^j_1-u_1(t_j)\right\|^2+\left\|U^j_2-u_2(t_j)\right\|^2\right).
\]

For the first component of the second term $\tilde{J}_2$, using Proposition \ref{prop1} we obtain
 \[
 \begin{split}
     \mathbb{E}\|\tilde{J}_{2,1}\|^{2}&=  \mathbb{E}\bigg\|\sum_{j=0}^{n-1}\int_{t_j}^{t_{j+1}}C_h(t_n-t_j)\mathcal{P}_h(f_1(u_1(t_j),u_2(t_j))-f_1(u_1(\tau),u_2(\tau)))dW_1(\tau)\\
     &-\sum_{j=0}^{n-1}\int_{t_j}^{t_{j+1}}S_h(t_n-t_j)\mathcal{P}_h(f_2(u_1(t_j),u_2(t_j))-f_2(u_1(\tau),u_2(\tau)))dW_2(\tau)\bigg\|^2\\
     &\leq 2\mathbb{E}\left\|\sum_{j=0}^{n-1}\int_{t_j}^{t_{j+1}}C_h(t_n-t_j)\mathcal{P}_h(f_1(u_1(t_j),u_2(t_j))-f_1(u_1(\tau),u_2(\tau)))dW_1(\tau)\right\|^2\\
     &+2\mathbb{E}\left\|\sum_{j=0}^{n-1}\int_{t_j}^{t_{j+1}}S_h(t_n-t_j)\mathcal{P}_h(f_2(u_1(t_j),u_2(t_j))-f_2(u_1(\tau),u_2(\tau)))dW_2(\tau)\right\|^2\\
    & =2\mathbb{E}\sum_{j=0}^{n-1}\int_{t_j}^{t_{j+1}}\left\|C_h(t_n-t_j)\mathcal{P}_h(f_1(u_1(t_j),u_2(t_j))-f_1(u_1(\tau),u_2(\tau)))Q^{1/2}_1\right\|^2_{HS}d(\tau)\\
    &+2\mathbb{E}\sum_{j=0}^{n-1}\int_{t_j}^{t_{j+1}}\left\|S_h(t_n-t_j)\mathcal{P}_h(f_2(u_1(t_j),u_2(t_j))-f_2(u_1(\tau),u_2(\tau)))Q^{1/2}_2\right\|^2_{HS}d(\tau)\\
    &\leq 2K_f\sum_{j=0}^{n-1}\int_{t_j}^{t_{j+1}}\mathbb{E}\left(\|u_1(t_j)-u_1(\tau)\|^2+\|u_2(t_j)-u_2(\tau)\|^2\right)d\tau\\
    &\leq 2K_fK_h\sum_{j=0}^{n-1}\int_{t_j}^{t_{j+1}}|t_j-\tau|^{\min(\theta ,1)}d\tau\\
    &\leq Ck^{\min(\theta,1)},
 \end{split}
 \]
 for $\theta \in [0,2]$ and $K_h$ is a H\"{o}lder constant. Above we have used the H\"{o}lder regularity of $u_1 \text{ and } u_2$ from Proposition \ref{prop1} and the value of $k$ less than one.\\

A similar estimate holds for the second component
\[
 \mathbb{E}\|\tilde{J}_{2,2}\|^{2}\leq Ck^{\min(\theta,1)}.
\]

For the first component of third term $\tilde{J}_3$ term , we use It\^{o} isometry and Theorem $1.3$ of \cite{finite} to get 

\[
\begin{split}
\mathbb{E}\|\tilde{J}_{3,1}\|^2&=\mathbb{E}\bigg\|\sum_{j=0}^{n-1}\int_{t_j}^{t_{j+1}}(C_h(t_n-t_j)\mathcal{P}_h-C(t_n-t_j))f_1(u_1(\tau),u_2(\tau))dW_1(\tau)\\
&- \sum_{j=0}^{n-1}\int_{t_j}^{t_{j+1}}(S_h(t_n-t_j)\mathcal{P}_h-S(t_n-t_j))f_2(u_1(\tau),u_2(\tau))dW_2(\tau)\bigg\|^2\\
&\leq 2\mathbb{E} \left\|\sum_{j=0}^{t_{j+1}}\int_{t_j}^{t_{j+1}}(C_h(t_n-t_j)\mathcal{P}_h-C(t_n-t_j))f_1(u_1(\tau),u_2(\tau))dW_1(\tau)\right\|^2\\
&+ 2\mathbb{E} \left\|\sum_{j=0}^{t_{j+1}}\int_{t_j}^{t_{j+1}}(S_h(t_n-t_j)\mathcal{P}_h-S(t_n-t_j))f_2(u_1(\tau),u_2(\tau))dW_2(\tau)\right\|^2\\
&= 2\mathbb{E}\sum_{j=0}^{n-1}\int_{t_j}^{j+1}\left\|(C_h(t_n-t_j)\mathcal{P}_h-C(t_n-t_j))f_1(u_1(\tau),u_2(\tau))Q^{1/2}_1\right\|^2_{HS}d\tau\\
&+2\mathbb{E}\sum_{j=0}^{n-1}\int_{t_j}^{j+1}\left\|(S_h(t_n-t_j)\mathcal{P}_h-S(t_n-t_j))f_2(u_1(\tau),u_2(\tau))Q^{1/2}_2\right\|^2_{HS}d\tau\\
&\leq \sum_{j=0}^{n-1}\int_{t_j}^{t_{j+1}} Ch^{2\theta }\bigg[\mathbb{E}\left\|f_1(u_1(\tau),u_2(\tau))Q^{1/2}_1\right\|^2_{HS}+\mathbb{E}\|f_2(u_1(\tau),u_2(\tau))Q^{1/2}_2\|^2_{HS}\bigg]d\tau\\
& \leq\sum_{j=0}^{n-1}\int_{t_j}^{t_{j+1}}Ch^{2\theta }\bigg[\mathbb{E}\|A^{\theta/2}f_1(u_1(\tau),u_2(\tau))Q^{1/2}_1\|^2_{HS}\\
&\hspace{3cm}+\mathbb{E}\|A^{\theta/2}f_2(u_1(\tau),u_2(\tau))Q^{1/2}_2\|^2_{HS}\bigg]d\tau\\
& \leq Ch^{2\theta}\sup_{0 \leq t \leq T}\mathbb{E}(\|u_1(t)\|^2_{\theta}+\|u_2(t)\|^2_{\theta}+1)
\end{split}
\]
Similarly for the second component we have
\[
\mathbb{E}\|\tilde{J}_{3,2}\|^2\leq Ch^{2\theta}\sup_{0 \leq t \leq T}\mathbb{E}(\|u_1(t)\|^2_{\theta}+\|u_2(t)\|^2_{\theta}+1)
\]
For the first component of the fourth term $\tilde{J}_4$, we get
\[
\begin{split}
  \mathbb{E}\|\tilde{J}_{4,1}\|^2 &\leq 2 \mathbb{E}\left\|\sum_{j=0}^{n-1}\int_{t_j}^{t_{j+1}}(C(t_n-t_j)-C(t_n-\tau))f_1(u_1(\tau),u_2(\tau))dW_1(\tau)\right\|^2\\
    &+ 2\mathbb{E}\left\|\sum_{j=0}^{n-1}\int_{t_j}^{t_{j+1}}(S(t_n-t_j)-S(t_n-\tau))f_2(u_1(\tau),u_2(\tau))dW_2(\tau)\right\|^2\\
    & =2 \mathbb{E}\sum_{j=0}^{n-1}\int_{t_j}^{t_{j+1}}\left\|(C(t_n-t_j)-C(t_n-\tau))f_1(u_1(\tau),u_2(\tau))Q_1^{1/2}\right\|^2_{HS}d\tau\\
    &+2 \mathbb{E}\sum_{j=0}^{n-1}\int_{t_j}^{t_{j+1}}\left\|(S(t_n-t_j)-S(t_n-\tau))f_2(u_1(\tau),u_2(\tau))Q_2^{1/2}\right\|^2_{HS}d\tau\\  
\end{split}
\]
\[
\begin{split}
    ~~~~~~~~~~~~~~~~~~~~~& \leq \sum_{j=0}^{n-1}\int_{t_j}^{t_{j+1}}2C|t_j-\tau|^{\theta}\bigg[\mathbb{E}\left\|A^{\theta/2}f_1(u_1(\tau),u_2(\tau))Q_1^{1/2}\right\|^2_{HS}\\
&\hspace{4cm}+\mathbb{E}\|A^{\theta/2}f_2(u_1(\tau),u_2(\tau))Q_2^{1/2}\|^2_{HS}\bigg]d\tau\\
& \leq CL_fk^{\theta }\sup_{0 \leq t \leq T}\mathbb{E}(\|u_1(t)\|^2_{\theta}+\|u_2(t)\|^2_{\theta}+1).
\end{split}
\]   
Similar to the way we get an estimate for the second component that is 
\[
\mathbb{E}\|\tilde{J}_{4,1}\|^2 \leq Ck^{\theta } \sup_{0 \leq t \leq T}\mathbb{E}(\|u_1(t)\|^2_{\theta}+\|u_2(t)\|^2_{\theta}+1).
\]
\end{proof}
By consolidating all the preceding estimates of the stochastic error, we obtain the expectation of the norm bound for the first component of $Err_s$,  
 given by:
\[
\mathbb{E}\|Err_{s,1}\|^2\leq C\left[h^{2\theta}+k^{min(\theta ,1)}+k \sum_{j=0}^{n-1}\mathbb{E}\left(\left\|U_1^j-u_1(t_j)\right\|^2+ \left\|U_2^j-u_2(t_j)\right\|^2\right)\right].
\]
Here, we have used the fact that $k < 1$ and applied the upper bound property from Proposition \ref{prop1}.

Similarly, the expectation of the norm bound for the second component of $Err_s$ is expressed as:
\[
\mathbb{E}\|Err_{s,2}\|^2\leq C\left[h^{2\theta}+k^{min(\theta ,1)}+k \sum_{j=0}^{n-1}\mathbb{E}\left(\left\|U_1^j-u_1(t_j)\right\|^2+ \left\|U_2^j-u_2(t_j)\right\|^2\right)\right].
\]
By combining the estimates of the initial error, deterministic error, and stochastic error, we obtain the following estimates for the first and second components for $\theta \in [0,2]$:
\[
\begin{split}
    \mathbb{E}\|U^n_1-u_1(t_n)\|^2\leq C\left[h^{2\theta}+k^{min(\theta ,1)}+k \sum_{j=0}^{n-1}\mathbb{E}\left(\left\|U_1^j-u_1(t_j)\right\|^2+ \left\|U_2^j-u_2(t_j)\right\|^2\right)\right] ,\\
   \mathbb{E}\|U^n_2-u_2(t_n)\|^2\leq C\left[h^{2\theta}+k^{min(\theta ,1)}+k \sum_{j=0}^{n-1}\mathbb{E}\left(\left\|U_1^j-u_1(t_j)\right\|^2+ \left\|U_2^j-u_2(t_j)\right\|^2\right)\right].  \end{split}
\]
From the above error bounds and the discrete Gronwall lemma, we obtain
\[
\begin{split}
   \mathbb{E}\left(\left\|U^n_1-u_1(t_n)\right\|^2+\|U^n_2-u_2(t_n)\|^2\right)& \leq C(h^{2\theta}+k^{min(\theta,1)})\\
   &+kC\sum_{j=0}^{n-1} \mathbb{E}\left(\left\|U^j_1-u_1(t_j)\right\|^2+ \mathbb{E}\|U^j_2-u_2(t_j)\|^2\right)\\
   & \leq  C(h^{2\theta}+k^{min(\theta,1)})e^{Ckn}\leq  C(h^{2\theta}+k^{min(\theta,1)}).
\end{split}
\]
Thus, the error estimate for the first component is given by
\[
\|U^n_1-u_1(t_n)\|_{L_2(\Omega,\dot{H}^0)}= \left(\mathbb{E}\left\|U^n_1-u_1(t_n)\right\|^2\right)^{1/2} \leq C(h^{\theta}+k^{min(\theta/2,1/2)}).
\]
Similarly, we obtain the error estimate for the second component.

	\section{Numerical Experiments}\label{S_4}
	
	\subsection{Numerical Example}
	We consider the following stochastic semilinear Schr\"{o}dinger equation in one spatial dimension.
	\begin{equation}\label{eqn_num1}
		\begin{split}
			&du+i \Delta u\,dt =(\sqrt{u_2^2+1}+i\,\sqrt{u_1^2+1})dt+\sqrt{u_1^2+1}dW_1 +i\sqrt{u_2^2+1}\, dW_2  \text{ in } ( 0,1)\times  ( 0,1), \\
			&u(t,0)=0=u(t,1), \quad t\in (0,1),\\
			&u(0,x)=\sin(2\pi x)+i\,x(1-x), \quad x\in ( 0,1),
		\end{split}
	\end{equation}
    where $u_1 \text{ and }u_2$ are real and imaginary part of the solution $u$ respectively.
    Using the full discretization scheme \eqref{eqn1.10} we perform numerical computation for the example \eqref{eqn_num1}. The output supports Theorem \ref{thm4}. Below we provide Figure \ref{fig1} and Figure \ref{fig2} which describe the rate of convergence for the example \eqref{eqn_num1}.

	Let $\{\lambda_j\}_{j=1}^\infty$ be eigen values of $A$ and we take $Q_1=Q_2=A^{-s},\,s\in\mathbb{R}.$ Then, we have for $i=1,2,$
	$$
	\|A ^{\theta /2}Q_i^{1/2}\|_{HS}^2=\|A ^{(\theta-s) /2}\|_{HS}^2=\sum_{j=1}^\infty\lambda_j^{\theta-s} 
	\approx \sum_{j=1}^\infty j^{\frac{2}{d}(\theta-s)},
	$$
	which is finite if and only if $\theta< s-\frac{d}{2},$ where $d$ is the dimension of the spatial domain $\mathcal{O}.$
	In the example above in \eqref{eqn_num1}, $d=1.$ Hence, we require $\theta<s-\frac{1}{2}.$

\begin{figure}[htbp]
    \centering
    \begin{minipage}{0.45\textwidth}
        \centering
        \includegraphics[width=\linewidth]{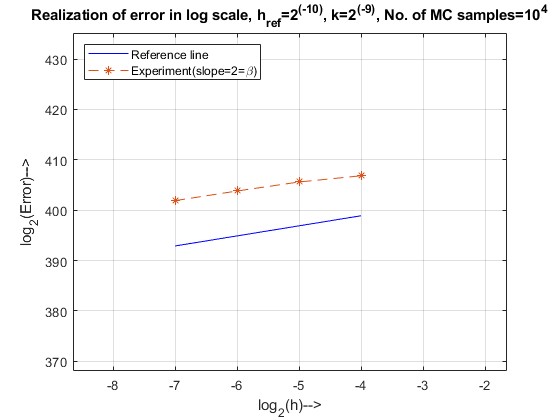} 
        \caption{ The rate of strong convergence with respect to the space discretized parameter h.}
        \label{fig1}
    \end{minipage}
    \hspace{0.05\textwidth} 
    \begin{minipage}{0.48\textwidth}
        \centering
        \includegraphics[width=\linewidth]{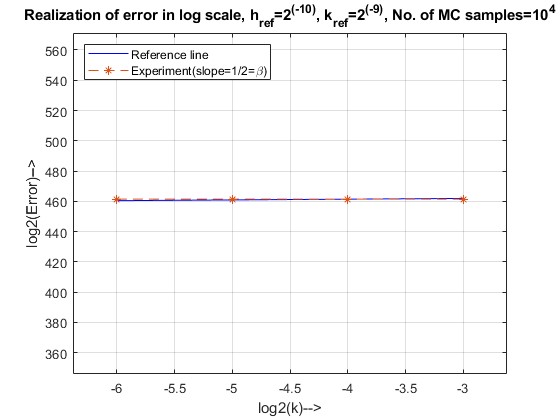} 
        \caption{The rate of strong convergence with respect to the time discretized parameter k.}
        \label{fig2}
    \end{minipage}
\end{figure}
		In the numerical experiment, we have considered two cases:
		\begin{enumerate}
			\item The rate of strong convergence of finite element approximations of stochastic semilinear  Schr\"{o}dinger equation with respect to the space discretized parameter $h$ : We take $\theta=2, d=1,$ hence, $s>2+\frac{1}{2}.$ We choose $s=2+\frac{1}{2}+0.001$, $h_{\text{ref}}=2^{-10}$(step length of space discretization for reference solution) and 
		$k=2^{-9}$(time step to be fixed) and $10^4$ Monte-Carlo samples for sampling in the stochastic case.
        see Figure \ref{fig1},
            \item The rate of strong convergence of time dicretized approximations of stochastic semilinear Schr\"{o}dinger equation with respect to the time discretized parameter $k$ : We take $\theta=.5, d=1,$ hence, $s>1.$ We choose $s=1+0.001$, $h=2^{-10}$(step length to be fixed ) and 
		$k_{\text{ref}}=02^{-9}$(step length of time discretization for reference solution) and $10^4$ Monte-Carlo samples for sampling in the stochastic case. see Figure \ref{fig2}.
		\end{enumerate}
	\begin{appendix}
		\section{ Notations and Definitions}\label{s2}
Let $(U_1,(\cdot,\cdot)_{U_1})\text{ and }(U_2,(\cdot,\cdot)_{U_2})$ be separable Hilbert spaces with corresponding norms $\|\cdot\|_{U_1} \text{ and } \|\cdot\|_{U_2}$. Let $\mathcal{L}(U_1,U_2)$ denote the space of bounded linear operators from $U_1 \text{ to } U_2$, and $\mathcal{L}_2(U_1,U_2)$ the space of Hilbert-Schmidt operators, endowed with norm $\|\cdot\|_{\mathcal{L}_2(U_1,U_2)}$. That is, $T_1 \in \mathcal{L}_2(U_1,U_2) $ if $T_1\in \mathcal{L}(U_1,U_2)$ and 
	\[
	\|T_1\|^2_{\mathcal{L}_2(U_1,U_2)}:=\sum_{j=1}^{\infty}\|T_1e_j\|_{U_2}^2< \infty,
	\]
	where $\{e_j\}_{j=1}^{\infty}$ is an orthonormal basis of $U_1$. The Hilbert-Schmidt norm does not depend on the choice of orthonormal basis. If $U_1=U_2$, we write $\mathcal{L}(U_1)=\mathcal{L}(U_1,U_1)$ and $HS= \mathcal{L}_2(U_1,U_1)$. It is a well-established fact that if
 $T_1 \in \mathcal{L}(U_1) \text{ and } T_2 \in \mathcal{L}_2(U_1,U_2)$, then their composition  $T_2T_1 \text{ also belongs to }\mathcal{L}_2(U_1,U_2)$. Moreover, the following norm bound holds:
	\[
	\|T_2T_1\|_{\mathcal{L}_2(U_1,U_2)} \leq \|T_2\|_{\mathcal{L}_2(U_1,U_2)}\|T_1\|_{\mathcal{L}(U_1)}.
	\]
	
	Let $\left(\Omega, \mathcal{F},P, \{\mathcal{F}_t\}_{t \geq 0}\right )$ be a filtered probability satisfying the usual conditions. We denote by $L^2(\Omega, U_2)$  the space consisting of square-integrable random variables that take values in $U_2$. The norm associated with this space is given by
	\[
	\|v\|_{L^2(\Omega,U_2)}=\mathbb{E}(\|v\|^2_{U_2})^{1/2}=\left(\int_{\Omega}\|v(\omega)\|_{U_2}^2 dP(\omega)\right)^{1/2},
	\]
	where $\mathbb{E}$ is the expectation. 
  \begin{definition}[\cite{rockner}]

   Consider an operator $Q \in \mathcal{L}(U_1)$ that is self-adjoint, positive, and semidefinite, satisfying $Tr(Q)< \infty$, where $Tr(Q)$ denotes its trace. A stochastic process $\{W(t)\}_{t \geq 0}$ taking values in $U_1$ is called  $Q$-Wiener process with respect to the filtration $\{\mathcal{F}_t\}_{t \geq 0}$ if it satisfies the following conditions:
	\begin{enumerate}
		\item $W(0)=0$  almost surely,
		\item $W$  has continuous sample paths with probability one,
		\item The increments of $W$ are independent,
		\item  For any $ 0\leq s \leq t$, the increment $W(t)-W(s), $  follows a Gaussian distribution in $U_1$ with mean zero and covariance operator
		 $(t-s)Q$.
	\end{enumerate}
\end{definition} 
Moreover the stochastic process  $\{W(t)\}_{t \geq 0}$ has the following representation
	\begin{equation}\label{eqn7}
		W(t)=\sum_{j=1}^{\infty} \gamma _j^{1/2} \beta_j(t)e_j,
	\end{equation}
	where $\{(\gamma_j,e_j)\}_{j=1}^{\infty}$ are the eigenpairs of $Q$, with $\{e_j\}$  forming an orthonormal basis, and
     $\{\beta _j\}_{j=1}^{\infty}$
	is a collection of independent standard Brownian motions in $\mathbb{R}.$ It can be verified that the series representation in \eqref{eqn7} converges in  $L^2(\Omega, U_1)$. Specifically, for any $t \geq 0$, we have
		\[
		\|W(t)\|_{L^2(\Omega,U_1)}^2= \mathbb{E}\left(\left\|\sum _{j=1} ^{\infty}\gamma_j ^{1/2} e_j \beta_j(t) \right\|_{U_1}^2\right)
	\]
Applying the linearity of expectation and orthonormality of $\{e_j\}$ this expression simplifies to
\[
    \sum_{j=1}^{\infty}\gamma_j\mathbb{E}(\beta_j(t))^2	 
    \]	
    Since each $\beta_j(t)$   is a standard Brownian motion, we use the property $\mathbb{E}\beta_j(t)^2=t$  to obtain
     \[
    \|W(t)\|_{L^2(\Omega,U_1)}^2= t \sum_{j=1}^{\infty}\gamma_j=t \text{Tr}(Q).
     \]   
        
 We consider here a specific instance of  It\^o's integral, where the integrand is a deterministic function.
	Let $\Phi : [0, \infty) \to  \mathcal{L}(U_1, U_2) $ be a strongly measurable function that satisfies the integrability condition
	\begin{equation}\label{eqn8}
		\int_0^t\|\Phi(\tau)Q^{1/2}\|^2_{\mathcal{L}_2(U_1, U_2) }d\tau < \infty.
	\end{equation}
	Under this assumption, the stochastic integral $\int_0^t\Phi(\tau)\,dW(\tau)$is well-defined, and It\^o's isometry isometry holds:
	\begin{equation} \label{eqn9}
		\left\|\int_0^t\Phi(\tau)\,dW(\tau)\right\|^2_{L^2(\Omega,U_2)}=\int_0^t \|\Phi (\tau) Q^{1/2}\|^2_{\mathcal{L}_2(U_1,U_2)}\,d\tau.
	\end{equation}
	
In a more general setting, consider a self-adjoint, positive, and semidefinite operator $Q \in \mathcal{L}(U_1)$  with eigenpairs $\{(\gamma_j, e_j )\}_{j=1}^{\infty}$. If $Q$ is not of trace class, meaning that  $Tr(Q) = \infty$,  then the series expansion given in \eqref{eqn7} fails to converge in $L^2(\Omega, U_1)$. However, it converges in a suitably chosen (usually larger) Hilbert space, and the stochastic integral $\int ^t _0 \Phi(\tau) dW(\tau) $ can still be defined, and the isometry \eqref{eqn9} holds, as long as \eqref{eqn8} is satisfied. In this case, $W$ is called a cylindrical Wiener process (\cite{prato}). In particular, we may have $Q = I$ (the identity operator).
    \section{Abstract framework and regularity}\label{s2.1}
    	Let us consider the Hilbert space $L^2(\mathcal{O})$ of square integrable Lebesgue measurable functions defined on $\mathcal{O}$ with the usual inner product $(\cdot,\cdot)$ and norm $\|\cdot\|.$
	Let us define the Laplace operator $({A}, D({A}))$ in $L^2(\mathcal{O})$
    as
    \begin{align*}
        D({A}):&=H^2(\mathcal{O}) \cap H^1_0(\mathcal{O}),\\
        {A}u :& =-\Delta u,\quad u
        \in D({A}).
    \end{align*} 
    From spectral analysis (see for e.g. \cite{Yosida}), we know that $({A}, D({A}))$ has eigenpairs $\{(\lambda_j,\phi_j)\}_{j=1}^{\infty}$ satisfying the following:
    $$
    0<\lambda_1< \lambda_2 \leq \lambda_3\leq \cdots,\quad \lim_{j\to\infty} \lambda_j=\infty \quad\text{and}\quad \|\phi_j\|=1,\,\forall j\geq 1.
    $$ We introduce the following fractional Sobolev space using the above spectral decomposition of $({A}, D({A}))$
    as 
	\begin{equation*}
		\dot{H} ^{\gamma}:=D({A} ^{\gamma/2}),\hspace{1cm}\|v\|_{\gamma}:=\|{A}^{\gamma/2}v\|=\left(\sum_{j=1}^{\infty}\lambda^{\gamma}_j(v,\phi_j)^2\right)^{1/2},\quad \gamma \in \mathbb{R},\quad v\in \dot{H} ^{\gamma}. 
	\end{equation*}
    We have $\dot{H} ^{\gamma_2} \subset \dot{H} ^{\gamma_1} \text{ for } \gamma_1 \leq \gamma_2.$ The spaces $ \dot{H} ^{-\gamma}$ can be identified with the dual space $( \dot{H} ^{\gamma})^{\star} \text{ for } \gamma >0;$ see\cite{Thomee}. We note that $\dot{H} ^{0}=L^2(\mathcal{O}),\text{ } \dot{H} ^{1}={H} ^{1}_0(\mathcal{O}), \text{ } \dot{H} ^{2}= D({A})=H^2(\mathcal{O}) \cap H^1_0(\mathcal{O})$ with equivalent norms . We  also observe that the inner product in $ \dot{H} ^{1} \text{ is }(\cdot,\cdot)_1=(\nabla  \cdot, \nabla \cdot).$ We introduce the product  Hilbert spaces for $ \gamma\in \mathbb{R}$
	\begin{equation}
{\mathbf{H}}^{\gamma}:=\dot{H}^{\gamma}\times\dot{H}^{\gamma},\hspace{1cm} |||\mathbf{v}|||_{\gamma}^2:=\|v_1\|^2_{\gamma}+\|v_2\|^2_{\gamma},\quad \mathbf{v}=(v_1,v_2)\in {\mathbf{H}}^{\gamma},
	\end{equation}
	and denote ${\mathbf{H}}={\mathbf{H}}^0=\dot{H}^{0}\times\dot{H}^{0}$ with corresponding norm $|||\cdot|||=|||\cdot|||_0$.
    For $s\in[-1,0]$, we define the operator $(\mathbb{A},D(\mathbb{A}))$ in ${\mathbf{H}}^{\gamma},$
    \begin{align*}
    D(\mathbb{A})&:={\mathbf{H}}^{s+2},\\
\mathbb{A}{\mathbf{u}}&:=\begin{bmatrix}
		0 & -A\\
		A & 0
	\end{bmatrix}\begin{bmatrix}u_1\\u_2	    
	\end{bmatrix}=\begin{bmatrix}
		- A u_2\\
		A u_1
	\end{bmatrix},\quad \mathbf{u}=(u_1,u_2)\in D(\mathbb{A}).
    \end{align*}
	The operator $(\mathbb{A},D(\mathbb{A}))$ generates a unitary group $\{e^{t\mathbb{A}}\}_{t\in\mathbb{R}}$ on  ${\mathbf{H}}^{\gamma}$ (see for e.g. \cite{apazy}) and it is given by
\begin{equation}\label{eqn6}
		e^{t\mathbb{A}}=\begin{bmatrix}
			C(t) & -S(t)\\
			S(t) & C(t)
		\end{bmatrix},\quad t\in\mathbb{R},
	\end{equation}
	where $C(t)=\cos{(tA)} \text{ and }S(t)=\sin{(tA})$ are the  cosine and sine operators. For example, using $ \{ (\lambda _j,\phi _j)\}_{j=1}^{\infty}$ the orthonormal eigenpairs of $A$, the cosine and sine operators are given as, for $ v\in \dot{H}^0$ and for $t\geq 0,$
	\begin{align*}
	C(t)v=\cos{(tA)}v=\sum_{j=1}^{\infty} \cos{(t\lambda_j)}(v,\phi _j)\phi_j,\\
	S(t)v=\sin{(tA)}v=\sum_{j=1}^{\infty} \sin{(t\lambda_j)}(v,\phi _j)\phi_j .
	\end{align*}
    \section{Finite element approximations}\label{s2.2}
Let $\mathcal{T}_h$ represent a regular family of triangulations of 
 the domain $\mathcal{O}$. The parameter $h$ is defined as  $$h=\max_{T\in \mathcal{T}_h}h_T,$$
 where  $h_T=\text{diam}(T)\text{ and, } T \text{ is triangle in }\mathcal{O}$. Define $V_h$ as the space of continuous, piecewise linear functions associated with $\mathcal{T}_h$ that vanish on $\partial \mathcal{O}$.  Consequently, we have $V_h \subset H^1_0(\mathcal{O})=\dot{H}^1.$ 
	
	Since $\mathcal{O}$  is assumed to be convex and polygonal, the triangulations can be precisely aligned with $\partial \mathcal{O}$. Moreover, this ensures the elliptic regularity property $\|w\|_{H^2(\mathcal{O})} \leq C\|Aw\|$ for $w \in D({A})$, see \cite{pgrisvard}. We now recall fundamental results from the finite element theory (see \cite{ciarlet, Brenner}). The norms are denoted as$\|\cdot\|_s=\|\cdot\|_{\dot{H}^s}$.
	
	Consider the orthogonal projection operators  $\mathcal{P}_h:\dot{H}^0 \to V_h\text{ and }\mathcal{R}_h:\dot{H}^1 \to V_h$ which are defined as
	\[
	(\mathcal{P}_hw,\psi)=(w, \psi),\hspace{.5cm} (\nabla \mathcal{R}_hw,\nabla \psi)=(\nabla w,\nabla \psi\
	)\quad \forall \psi \in V_h.
	\]
	
	To introduce a discrete analogue of the norm 
$\|\cdot\|_{s}$, we define 
	\[
	\|w_h\|_{h,s}=\|A _h^{s/2}w\|, \hspace{.3cm} w_h\in V_h, s \in \mathbb{R},
	\]
	where the discrete Laplacian  $A_h: V_h \to V_h$  
  is given by
	\[
	(A_hw_h,\psi)=(\nabla w_h,\nabla \psi) \hspace{.2cm} \forall \psi \in V_h.
	\]
We have the following relation between operators $A_h \text{ and } A$ (see Theorem 4.4 in \cite{kovacsbit(2012)})
    \begin{equation}\label{eqn2.4}
\|A^{\gamma}_h \mathcal{P}_h A ^{-\gamma}v\|_{L_2(\mathcal{O})}^2 \leq \|v\|^2_{L_2(\mathcal{O})}, \hspace{.2cm}\gamma \in [-\frac{1}{2},1], \hspace{.2cm } v \in \dot{H}^0=L_2(\mathcal{O})
\end{equation}
    \end{appendix}
	
	\medskip
	\noindent
	{\bf{\Large{Acknowledgement:}}} We express our gratitude to the Department of Mathematics and
	Statistics at the Indian Institute of Technology Kanpur for providing a conductive research
	environment. For this work,  M. Prasad is grateful
	for the support received through MHRD, Government of India (GATE fellowship). M. Biswas acknowledges support from IIT Kanpur. S. Bhar acknowledges the support from the SERB MATRICS grant (MTR/2021/000517), Government of India.
	
	\bibliographystyle{plain}

\begin{thebibliography}{10}

\bibitem{cohen2016}
Rikard Anton, David Cohen, Stig Larsson, and Xiaojie Wang.
\newblock Full discretization of semilinear stochastic wave equations driven by multiplicative noise.
\newblock {\em SIAM J. Numer. Anal.}, 54(2):1093--1119, 2016.

\bibitem{finite}
Suprio Bhar, Mrinmay Biswas, and Mangala Prasad.
\newblock Finite element approximations of stochastic linear {S}chr\"odinger equation driven by additive wiener noise.
\newblock {\em arXiv preprint arXiv:2410.06006}, 2024.

\bibitem{Brenner}
Susanne~C. Brenner and L.~Ridgway Scott.
\newblock {\em The mathematical theory of finite element methods}, volume~15 of {\em Texts in Applied Mathematics}.
\newblock Springer, New York, third edition, 2008.

\bibitem{courantL}
Thierry Cazenave.
\newblock {\em Semilinear {S}chr\"odinger equations}, volume~10 of {\em Courant Lecture Notes in Mathematics}.
\newblock New York University, Courant Institute of Mathematical Sciences, New York; American Mathematical Society, Providence, RI, 2003.

\bibitem{ciarlet}
Philippe~G. Ciarlet.
\newblock {\em The finite element method for elliptic problems}, volume~40 of {\em Classics in Applied Mathematics}.
\newblock Society for Industrial and Applied Mathematics (SIAM), Philadelphia, PA, 2002.
\newblock Reprint of the 1978 original [North-Holland, Amsterdam; MR0520174 (58 \#25001)].

\bibitem{davidcohen16}
David Cohen and Llu\'is Quer-Sardanyons.
\newblock A fully discrete approximation of the one-dimensional stochastic wave equation.
\newblock {\em IMA J. Numer. Anal.}, 36(1):400--420, 2016.

\bibitem{prato}
Giuseppe Da~Prato and Jerzy Zabczyk.
\newblock {\em Stochastic equations in infinite dimensions}, volume~44 of {\em Encyclopedia of Mathematics and its Applications}.
\newblock Cambridge University Press, Cambridge, 1992.

\bibitem{dautray}
Robert Dautray and Jacques-Louis Lions.
\newblock {\em Mathematical analysis and numerical methods for science and technology. {V}ol. 5}.
\newblock Springer-Verlag, Berlin, 1992.
\newblock Evolution problems. I, With the collaboration of Michel Artola, Michel Cessenat and H\'el\`ene Lanchon, Translated from the French by Alan Craig.

\bibitem{debussche1999}
A.~de~Bouard and A.~Debussche.
\newblock A stochastic nonlinear {S}chr\"odinger equation with multiplicative noise.
\newblock {\em Comm. Math. Phys.}, 205(1):161--181, 1999.

\bibitem{debussche2003}
A.~de~Bouard and A.~Debussche.
\newblock The stochastic nonlinear {S}chr\"odinger equation in {$H^1$}.
\newblock {\em Stochastic Anal. Appl.}, 21(1):97--126, 2003.

\bibitem{pgrisvard}
Pierre Grisvard.
\newblock {\em Elliptic problems in nonsmooth domains}, volume~69 of {\em Classics in Applied Mathematics}.
\newblock Society for Industrial and Applied Mathematics (SIAM), Philadelphia, PA, 2011.
\newblock Reprint of the 1985 original [MR0775683], With a foreword by Susanne C. Brenner.

\bibitem{kovacsbit(2012)}
Mih\'aly Kov\'acs, Stig Larsson, and Fredrik Lindgren.
\newblock Weak convergence of finite element approximations of linear stochastic evolution equations with additive noise.
\newblock {\em BIT}, 52(1):85--108, 2012.

\bibitem{apazy}
A.~Pazy.
\newblock {\em Semigroups of linear operators and applications to partial differential equations}, volume~44 of {\em Applied Mathematical Sciences}.
\newblock Springer-Verlag, New York, 1983.

\bibitem{rockner}
Claudia Pr\'ev\^ot and Michael R\"ockner.
\newblock {\em A concise course on stochastic partial differential equations}, volume 1905 of {\em Lecture Notes in Mathematics}.
\newblock Springer, Berlin, 2007.

\bibitem{Thomee}
Vidar Thom\'{e}e.
\newblock {\em Galerkin finite element methods for parabolic problems}, volume~25 of {\em Springer Series in Computational Mathematics}.
\newblock Springer-Verlag, Berlin, second edition, 2006.

\bibitem{YYanSiam5}
Yubin Yan.
\newblock Galerkin finite element methods for stochastic parabolic partial differential equations.
\newblock {\em SIAM J. Numer. Anal.}, 43(4):1363--1384, 2005.

\bibitem{Yosida}
K.~Yosida.
\newblock {\em Functional analysis}.
\newblock Classics in Mathematics. Springer-Verlag, Berlin, 1995.
\newblock Reprint of the sixth (1980) edition.

\end{thebibliography}

\end{document}